\documentclass[a4paper,10pt]{amsart}
\usepackage{pst-all, pstricks, pst-node}
\usepackage{subfigure}
\usepackage{graphicx, amssymb, amsmath, amscd, hhline}
\usepackage[arrow,tips,matrix]{xy}

\theoremstyle{ams}
\newtheorem{theorem}{Theorem}[section]
\newtheorem{proposition}[theorem]{Proposition}
\newtheorem{lemma}[theorem]{Lemma}
\newtheorem{corollary}[theorem]{Corollary}

\theoremstyle{definition}

\newtheorem{definition}[theorem]{Definition}
\newtheorem{remark}[theorem]{Remark}

\newtheorem{example}[theorem]{Example}

\numberwithin{table}{section}
\numberwithin{figure}{section}
\numberwithin{equation}{section} 

\DeclareMathOperator{\ev}{ev}
\DeclareMathOperator{\id}{id}
\DeclareMathOperator{\D}{D}
\DeclareMathOperator{\GL}{GL}

\DeclareMathOperator{\im}{im}
\DeclareMathOperator{\res}{res}
\DeclareMathOperator{\iso}{iso}
\DeclareMathOperator{\fiso}{fiso}
\DeclareMathOperator{\End}{end}

\DeclareMathOperator{\gl}{gl}

\DeclareMathOperator{\ext}{ext}

\DeclareMathOperator{\Id}{Id}

\newcommand{\Cl}{\mathrm{Cl}}
\newcommand{\cl}{\mathrm{cl}}

\newcommand{\tetra}{\mathrm{T}}

\newcommand{\field}[1]{\mathbb{#1}}
\newcommand{\C}{\field{C}}

\newcommand{\Z}{\field{Z}}
\newcommand{\N}{\field{N}}

\begin{document}

\title{Equivariant pointwise clutching maps}

\author[M. K. Kim]{Min Kyu Kim}
\address{Department of Mathematics Education,
Gyeongin National University of Education, 45
Gyodae-Gil, Gyeyang-gu, Incheon, 407-753,
Republic of Korea} \email{mkkim@kias.re.kr}


\maketitle

%

\begin{abstract}
In the paper, we introduce the terminology
\textit{equivariant pointwise clutching map}. By
using this, we give details on how to glue an
equivariant vector bundle over a finite set so as
to obtain a new Lie group representation such
that the quotient map from the bundle to the
representation is equivariant. Then, we
investigate the topology of the set of all
equivariant pointwise clutching maps with respect
to an equivariant vector bundle over a finite
set. Results of the paper play a key role in
classifying equivariant vector bundles over
two-surfaces in other papers.
\end{abstract}

\maketitle

\section{Introduction} \label{section: introduction}

Though equivariant vector bundle is a usual
object in topology and geometry, there are only a
few classification results, for example in the
following extreme cases:
\begin{itemize}
\item a group action on a base space is free, see
\cite[p. 36]{At}, \cite[p. 132]{S},
\item a group action on a base space is transitive,
see \cite[p. 130]{S}, \cite[Proposition II.3.2]{B},
\item a base space is $S^1,$ see \cite{CKMS}.
\end{itemize}
So, it would be meaningful even if we develop a
systematic approach to classification of
equivariant vector bundles only over
low-dimensional manifolds. In oncoming papers,
the author will do so through a slight
generalization of clutching construction, which
he calls equivariant clutching construction, see
\cite[p. 23$\sim$24]{At} for clutching
construction. And in this paper, we will
investigate its technical essence (or its
pointwise version) by using the terminology
equivariant pointwise clutching map. So, we need
simply explain the construction as a motivation.
Before we give an explanation by an example, we
introduce a useful notation. Let $G$ be a compact
Lie group acting continuously on a topological
space $X,$ and let $E$ be an equivariant
topological complex vector bundle over $X.$ For a
subset $A$ in $X,$ let $G_A$ be the maximal
subgroup of $G$ preserving $A.$ Then, $E|_A$ is
preserved by the group action restricted to
$G_A,$ where $E|_A$ is the pull-back bundle of
$E$ by the inclusion from $A$ to $X.$ Denote the
bundle $E|_A$ equipped with the $G_A$-vector
bundle structure by $E_A.$ For a one point subset
$\{ x \} \subset X,$ we denote $E_{\{ x \}}$
simply by $E_x,$ which is just a
$G_x$-representation.

\begin{figure}[ht]
\begin{center}
\begin{pspicture}(-4,0)(5,3) \footnotesize

\pspolygon[fillstyle=solid,fillcolor=lightgray](1.5,
1)(3.25, 0.75)(4,1.25)(2.75,2.5)(1.5, 1)
\psline[linestyle=dotted](1.5, 1)(4,1.25)
\psline(3.25, 0.75)(2.75,2.5)

\pspolygon[fillstyle=solid,fillcolor=lightgray](-1.5,0.75)(-0.75,1.25)(-3.25,1)(-1.5,0.75)
\pspolygon[fillstyle=solid,fillcolor=lightgray](-0.75,1.5)(-2,2.75)(-3.25,1.25)(-0.75,1.5)
\pspolygon[fillstyle=solid,fillcolor=lightgray](-1,
0.75)(-0.25,1.25)(-1.5,2.5)(-1,0.75)
\pspolygon[fillstyle=solid,fillcolor=lightgray](-2,0.75)(-2.5,2.5)(-3.75,1)(-2,0.75)

\psline[arrowsize=5pt]{->}(0.25,1.75)(1.25,1.75)

\uput[u](0.75,1.75){$\pi$} \uput[u](3,0){$X$}
\uput[u](-1.75,0){$\overline{X}$}

\end{pspicture}
\end{center}
\caption{ \label{figure: tetrahedron} The
disjoint union of four faces of a regular
tetrahedron}
\end{figure}
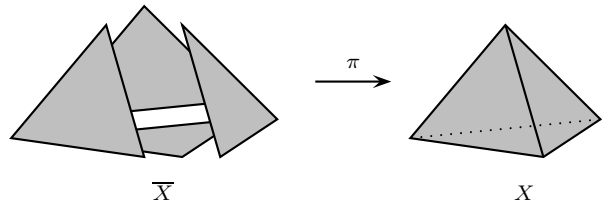

\begin{example} \label{example: tetrahedron}
Denote by $\tetra$ the order 24 full isometry
group of a regular tetrahedron $X.$ Then,
$\tetra$ acts naturally on $X.$ In addition to
the regular tetrahedron, we also consider the
disjoint union $\overline{X}$ of four faces of
$X,$ and the quotient map $\pi : \overline{X}
\rightarrow X$ which is the identity map whenever
restricted to each face, see Figure \ref{figure:
tetrahedron}. Then, the $\tetra$-action on $X$ is
lifted to $\overline{X},$ i.e. $\overline{X}$ is
equipped with a unique group action such that
$\pi$ is equivariant. Here, $X$ can be considered
to be obtained from $\overline{X}$ by gluing such
that the quotient map $\pi$ is equivariant. We
can say similar things to an equivariant
topological complex vector bundle $E$ over $X.$
The $\tetra$-action on $E$ is lifted to the
pull-back bundle $\pi^* E$ of $E$ by $\pi.$ And,
$E$ can be considered to be obtained from $\pi^*
E$ by gluing such that the quotient map from
$\pi^* E$ to $E$ is equivariant. The author calls
this way of constructing $E$ \textit{equivariant
clutching construction}. Here, it is conceivable
that we could obtain various equivariant vector
bundles over $X$ by different ways of gluing.
From this viewpoint, we only have to investigate
equivariant vector bundles over $\overline{X}$
and their possible ways of gluing so as to
classify equivariant vector bundles over $X.$
However since equivariant vector bundles over
$\overline{X}$ are easily understood, we only
have to focus on possible ways of gluing. In
doing so, the first thing to do is fiberwise
gluing. To explain this, assume for the moment
that we have already succeeded in obtaining a new
equivariant vector bundle $K$ over $X$ from an
equivariant vector bundle $J$ over $\overline{X}$
by gluing such that the quotient map from $J$ to
$K$ is equivariant. Then, this implies that for
each $x$ in some edge of $X,$ we have obtained a
new $G_x$-representation $K_x$ from the
$G_x$-vector bundle $J_{\pi^{-1}(x)}$ by gluing
such that the restricted quotient map from
$J_{\pi^{-1}(x)}$ to $K_x$ is equivariant. We
might call such a gluing of an equivariant vector
bundle over a finite set \textit{fiberwise
gluing} temporarily.  \qed
\end{example}

In the paper, we are concerned about fiberwise
gluing. Our investigation of it is divided into
two steps in a general setting as follows: given
a compact Lie group $G$ acting continuously on a
finite set $\bar{\mathbf{x}}$ and an equivariant
topological complex vector bundle $F$ over
$\bar{\mathbf{x}},$ we investigate (1) how to
glue actually $F$ so as to obtain a new
$G$-representation $w$ such that the quotient map
from $F$ to $w$ is equivariant, and then (2) the
topology of the set of all possible ways to glue
$F.$

First, we deal with (1). To glue $F,$ it is
natural to consider a map in
\begin{equation} \label{equation: product}
\prod_{(\bar{x}, \bar{x}^\prime) \in
\bar{\mathbf{x}} \times \bar{\mathbf{x}}} ~ \iso
\big( F_{\bar{x}}, ~ F_{\bar{x}^\prime} \big),
\end{equation}
i.e. a map $\psi$ contained therein is defined on
$\bar{\mathbf{x}} \times \bar{\mathbf{x}}$ and
the image $\psi (\bar{x}, \bar{x}^\prime)$ of
$(\bar{x}, \bar{x}^\prime) \in \bar{\mathbf{x}}
\times \bar{\mathbf{x}}$ is contained in $\iso
\big( F_{\bar{x}}, ~ F_{\bar{x}^\prime} \big),$
where the notation $\iso(\cdot, \cdot)$ denotes
the set of inequivariant(i.e. not considering a
group action) isomorphisms between two vector
spaces. Such a map $\psi$ is called a
\textit{pointwise clutching map} with respect to
$F$ if it satisfies the following:
\begin{enumerate}
\item[(a)] reflexivity : $~\psi(\bar{x}, \bar{x}) = \id_{F_{\bar{x}}},$
\item[(b)] symmetry : $~\psi(\bar{x}, \bar{x}^\prime) =
\psi(\bar{x}^\prime, \bar{x})^{-1},$
\item[(c)] transitivity : $~\psi(\bar{x}, \bar{x}^{\prime \prime}) =
\psi(\bar{x}^\prime, \bar{x}^{\prime \prime})
\psi(\bar{x}, \bar{x}^\prime)$
\end{enumerate}
for any $\bar{x}, ~ \bar{x}^\prime, ~
\bar{x}^{\prime \prime} \in \bar{\mathbf{x}},$
where $\id_V$ for a vector space $V$ is the
identity map on $V.$ By using a pointwise
clutching map $\psi,$ we would glue $F$ by
defining an equivalence relation $\sim$ on $F$ as
follows:
\begin{equation*}
\bar{u} \sim \bar{u}^\prime ~ \Longleftrightarrow
~ \psi(\bar{x}, \bar{x}^\prime) \bar{u} =
\bar{u}^\prime \qquad \text{for any } \bar{x},~
\bar{x}^\prime \in \bar{\mathbf{x}} \text{ and }
\bar{u} \in F_{\bar{x}}, ~ \bar{u}^\prime \in
F_{\bar{x}^\prime},
\end{equation*}
i.e. $\bar{u}$ and $\bar{u}^\prime$ are glued if
and only if $\bar{u} \sim \bar{u}^\prime.$ Denote
by $F/\psi$ the quotient space of $F$ by the
equivalence relation, and let $p_\psi : F
\rightarrow F/\psi$ be the quotient map. Since
$\psi$ is isomorphism-valued, the quotient space
$F/\psi$ inherits a unique vector space structure
from $F$ such that $p_\psi$ is a
\textit{fiberwise isomorphism}, i.e. an
inequivariant isomorphism between vector spaces
whenever restricted to each fiber. So, we have
obtained a vector space $F/\psi$ through gluing
$F$ by using a pointwise clutching map $\psi.$ To
equip the vector space $F/\psi$ with an
additional $G$-representation structure to match
our goal (1), the pointwise clutching map $\psi$
should satisfy an additional equivariance
condition.

We introduce a $G$-action on the set of pointwise
clutching maps with respect to $F.$ For a
pointwise clutching map $\psi$ with respect to
$F,$ if we define a new map $g \cdot \psi$ for $g
\in G$ in the product (\ref{equation: product})
as follows:
\begin{equation*}
(g \cdot \psi)(\bar{x}, \bar{x}^\prime) \bar{u} ~
:= ~ g \psi( g^{-1} \bar{x}, g^{-1}
\bar{x}^\prime ) g^{-1} \bar{u} \qquad \text{for
} \bar{x}, ~ \bar{x}^\prime \in \bar{\mathbf{x}}
\text{ and } \bar{u} \in F_{\bar{x}},
\end{equation*}
then $g \cdot \psi$ is also a pointwise clutching
map with respect to $F.$ If $\psi$ is fixed by
the action, i.e. $g \cdot \psi = \psi$ for each
$g \in G,$ then $\psi$ is called
\textit{equivariant}. In Lemma \ref{lemma:
equivalent condition for psi}, it is proved that
a pointwise clutching map $\psi$ with respect to
$F$ is equivariant if and only if the vector
space $F / \psi$ can be equipped with a unique
$G$-representation structure such that $p_\psi$
is equivariant. The unique $G$-representation
structure is called a \textit{glued
representation}, and denoted simply by the same
notation $F / \psi.$ In summary, we glue $F$ by
using an equivariant pointwise clutching map
$\psi$ so as to obtain the glued representation
$F/\psi$ so that the quotient map $p_\psi$ is
equivariant. This is our answer to (1).

Next, we deal with (2). Let $\Psi_F$ be the set
of all equivariant pointwise clutching maps with
respect to $F.$ We topologize $\Psi_F$ with the
subspace topology inherited from the product
topology on the product space (\ref{equation:
product}). In \cite{K}, the author classified
equivariant vector bundles over two-sphere
through equivariant clutching construction, and
from this he guesses affirmatively that a similar
result holds for other two-surfaces. Technical
difficulty therein is to calculate the zeroth
homotopy group of $\Psi_F$ and the first homotopy
group of each path-component of $\Psi_F$ for
suitable bundles $F$'s in our notations,
especially when the $G$-action on
$\bar{\mathbf{x}}$ is transitive. So in this
paper, we will restrict our interest to $\pi_0$
and $\pi_1$ of $\Psi_F.$ But, if readers are not
satisfied with two-surfaces and want to classify
equivariant vector bundles over higher
dimensional manifolds, then they should calculate
higher homotopy groups of $\Psi_F$ for suitable
$F$'s. Though we do not indulge into them, we
provide some needed tools, see Theorem
\ref{theorem: homeo type of Psi} and Section
\ref{section: fundamental group}.

Now, we introduce a terminology and two maps to
state a result on $\pi_0 ( \Psi_F ).$ A complex
$G$-representation $w$ is called a
\textit{representation extension} of $F$ if
$\res_{G_{\bar{x}}}^G w \cong F_{\bar{x}}$ for
each $\bar{x} \in \bar{\mathbf{x}}.$ And, we
denote by $\ext F$ the set of $G$-isomorphism
classes of all representation extensions of $F.$
For example, the glued representation $F/\psi$
for $\psi \in \Psi_F$ is contained in $\ext F$
because $p_\psi$ is an equivariant fiberwise
isomorphism. So, we can define the following
maps:
\begin{alignat*}{2}
\gl &: \Psi_F \longrightarrow \ext F,
& \qquad \psi \mapsto F/\psi & \qquad \text{and} \\
\pi_0(\gl) &: \pi_0(\Psi_F) \longrightarrow \ext
F, & \qquad [\psi] \mapsto F/\psi, &
\end{alignat*}
where the notation $\gl$ is the first two letters
of glued representation and the second map is
well-defined, see arguments below Lemma
\ref{lemma: become a representation} for detail.

\begin{theorem}
\label{theorem: bijectivity with extensions}
Let a compact Lie group $G$ act continuously on a
finite set $\bar{\mathbf{x}},$ and let $F$ be an
equivariant topological complex vector bundle
over $\bar{\mathbf{x}}.$ Then, the map
$\pi_0(\gl)$ is bijective. Especially, $\Psi_F$
is nonempty if and only if $\ext F$ is nonempty.
\end{theorem}

In the next theorem, we express each
path-component of $\Psi_F$ as the quotient space
of a Lie group by a closed subgroup, and this
enables us to calculate homotopy groups of each
path-component of $\Psi_F.$ For $w \in \ext F,$
denote by $\Psi_F (w)$ the preimage $\gl^{-1}
(w).$ Then, each $\Psi_F (w)$ is a nonempty
path-component of $\Psi_F$ by Proposition
\ref{proposition: cl is homeomorphic}.

\begin{theorem}
\label{theorem: homeo type of Psi} Under the
assumption of Theorem \ref{theorem: bijectivity
with extensions}, let $\bar{\mathbf{s}} \subset
\bar{\mathbf{x}}$ be a subset containing exactly
one element in each orbit of $\bar{\mathbf{x}}.$
Then, the path-component $\Psi_F (w)$ for each $w
\in \ext F$ is homeomorphic to the quotient space
\[
\Bigg( \prod_{{\bar{s}} \in \bar{\mathbf{s}}} ~
\iso_{G_{{\bar{s}}}} ( w ) \Bigg) \Bigg/ \iso_G (
w ),
\]
where $\iso_K (w)$ for a closed subgroup $K
\subset G$ is the group of all $K$-isomorphisms
of $w$ and the group $\iso_G (w)$ is diagonally
imbedded into the product group.
\end{theorem}


Through the homeomorphism of Theorem \ref{theorem:
homeo type of Psi}, we can show that $\Psi_F (w)$ is
simply connected in some important cases.

\begin{proposition} \label{proposition: simply connected}
Under the assumption of Theorem \ref{theorem:
bijectivity with extensions}, we assume
additionally that the $G$-action on
$\bar{\mathbf{x}}$ is transitive. Given a
representation extension $w$ of $F$ and a point
$\bar{x} \in \bar{\mathbf{x}},$ if
$\res_{G_{\bar{x}}}^G U$ is irreducible for each
irreducible $G$-representation $U$ contained in
$w,$ then $\Psi_F (w)$  is simply connected.
\end{proposition}

\begin{corollary} \label{corollary: simply connected}
Under the assumption of Theorem \ref{theorem:
bijectivity with extensions}, if the $G$-action on
$\bar{\mathbf{x}}$ is transitive and $G$ is abelian,
$\Psi_F (w)$ is simply connected for each $w \in \ext
F.$
\end{corollary}

This paper is organized as follows. In Section
\ref{section: proof}, we prove Theorem
\ref{theorem: bijectivity with extensions},
\ref{theorem: homeo type of Psi}. Section
\ref{section: fundamental group} deals with the
fundamental group of the set of equivariant
pointwise clutching maps, and we prove
Proposition \ref{proposition: simply connected}
and Corollary \ref{corollary: simply connected}.
In Section \ref{section: lemmas}, we prove two
useful lemmas on evaluation and restriction of an
equivariant pointwise clutching map.

\bigskip

\section{Proofs} \label{section: proof}
Hereafter, let a compact Lie group $G$ act
continuously on a finite set $\bar{\mathbf{x}},$
and let $F$ be an equivariant topological complex
vector bundle over $\bar{\mathbf{x}}$ as before.
Also, we abbreviate finite-dimensional
representation of a Lie group as representation.

In this section, we prove main theorems. We begin
this section by introducing a terminology. For a
complex $G$-representation $w,$ we denote by
$\fiso_G (F, w)$ the set of all equivariant
fiberwise isomorphisms from $F$ to $w.$ Below
Lemma \ref{lemma: define clutching}, we will see
that the set $\Psi_F ( w )$ for $w \in \ext F$ is
highly related to the set $\fiso_G (F, w).$ Since
$\fiso_G (F, w)$ is relatively easy to deal with,
our strategy for proofs is to investigate
$\fiso_G (F, w)$ first and then to relate
$\fiso_G (F, w)$ to $\Psi_F (w).$

To begin with, we prove a basic lemma on
equivariant pointwise clutching maps.

\begin{lemma} \label{lemma: equivalent condition for psi}
A pointwise clutching map $\psi$ with respect to
$F$ is equivariant if and only if the vector
space $F/\psi$ can be equipped with a complex
$G$-representation structure such that $p_\psi$
is equivariant. If exists, such a
$G$-representation structure is unique.
\end{lemma}

\begin{proof}
The unique possible group action on $F/\psi$ to
guarantee equivariance of $p_\psi$ is as follows:
\begin{equation*}
g \cdot u ~:=~ p_\psi ( g \bar{u} ) \quad \text{ for }
g \in G, ~u \in F/\psi, \text{ and any } \bar{u} \in F
\text{ such that } u=p_\psi ( \bar{u} ).
\end{equation*}
This is because if there exists such an action, it
should satisfy the following:
\begin{alignat*}{2}
p_\psi ( g \bar{u} ) &= g \cdot p_\psi ( \bar{u}
) & & \qquad \text{by equivariance of } p_\psi \\
&= g \cdot u. & & \qquad \text{by definition of }
\bar{u}
\end{alignat*}
It is easy that if the action is well-defined,
then it is actually an action. By definition, the
possible group action is well-defined if and only
if
\begin{equation}
\tag{*} p_\psi ( \bar{u} ) = p_\psi (
\bar{u}^\prime ) \quad \Longrightarrow \quad
p_\psi ( g \bar{u} ) = p_\psi ( g \bar{u}^\prime
)
\end{equation}
for any $g \in G$ and $\bar{u}, \bar{u}^\prime
\in F.$ By definition of $p_\psi,$ each side of
(*) is rewritten as follows:
\begin{align*}
p_\psi ( \bar{u} ) = p_\psi ( \bar{u}^\prime )
\quad &\Longleftrightarrow \quad \psi(\bar{x},
\bar{x}^\prime) \bar{u} =
\bar{u}^\prime \qquad \text{and}  \\
p_\psi ( g \bar{u} ) = p_\psi ( g \bar{u}^\prime
) \quad &\Longleftrightarrow \quad \psi(g
\bar{x}, g \bar{x}^\prime) g \bar{u} = g
\bar{u}^\prime, \quad \text{ i.e. } (g^{-1} \cdot
\psi)(\bar{x}, \bar{x}^\prime) \bar{u} =
\bar{u}^\prime,
\end{align*}
when $\bar{u} \in F_{\bar{x}}$ and
$\bar{u}^\prime \in F_{\bar{x}^\prime}$ for some
$\bar{x}, ~ \bar{x}^\prime \in \bar{\mathbf{x}}.$
So, (*) is rewritten as
\begin{equation*}
\psi(\bar{x}, \bar{x}^\prime) \bar{u} =
\bar{u}^\prime \quad \Longrightarrow \quad
(g^{-1} \cdot \psi)(\bar{x}, \bar{x}^\prime)
\bar{u} = \bar{u}^\prime
\end{equation*}
for any $g \in G$ and $\bar{u}, \bar{u}^\prime
\in F.$ In summary, the unique possible action is
well-defined if and only if $\psi = g^{-1} \cdot
\psi$ for any $g \in G,$ i.e. $\psi$ is
equivariant. Therefore, we obtain a proof.
\end{proof}

When $\psi$ is equivariant, the unique
$G$-representation structure on $F / \psi$ such
that $p_\psi$ is equivariant is denoted simply by
the same notation $F / \psi.$ In this case,
$p_\psi$ is contained in $\fiso_G (F, F / \psi),$
and is the prototype of equivariant fiberwise
isomorphisms.

For a time being, we investigate $\fiso_G (F,
w).$ We start by an easy lemma.

\begin{lemma} \label{lemma: become extension}
If there exists an equivariant fiberwise
isomorphism $p : F \rightarrow w$ for a complex
$G$-representation $w,$ then $w$ is contained in
$\ext F.$
\end{lemma}

Next, we factorize $\fiso_G (F, w)$ into a
product of isomorphism groups. By using this, we
can prove the converse of Lemma \ref{lemma:
become extension}, i.e. for any $w$ in $\ext F,$
there exists an equivariant fiberwise isomorphism
from $F$ to $w.$

\begin{proposition} \label{proposition: homeo type of fiso}
Let $\bar{\mathbf{s}} \subset \bar{\mathbf{x}}$
be a subset containing exactly one element in
each orbit of $\bar{\mathbf{x}}.$ For a complex
$G$-representation $w,$ the following map is
homeomorphic:
\[
\imath : \fiso_G (F, w) ~ \longrightarrow ~
\prod_{{\bar{s}} \in \bar{\mathbf{s}}} ~
\iso_{G_{\bar{s}}} ( F_{{\bar{s}}}, w ), \quad p
~ \longmapsto ~ \prod_{{\bar{s}} \in
\bar{\mathbf{s}}} ~ p|_{F_{{\bar{s}}}}.
\]
And, the set $\fiso_G (F, w)$ is nonempty if and
only if $w \in \ext F.$
\end{proposition}

\begin{proof}
First, we describe $F$ precisely. Since
\begin{equation*}
\bar{\mathbf{x}} = \coprod_{\bar{s} \in
\bar{\mathbf{s}}} G \cdot \bar{s},
\end{equation*}
the equivariant topological vector bundle $F$ can be
regarded as the disjoint union of $F$ restricted to
each orbit as follows:
\begin{equation*}
F = \coprod_{\bar{s} \in \bar{\mathbf{s}}} F_{G \cdot
\bar{s}},
\end{equation*}
where $\coprod$ is the notation for disjoint
union. By these expressions, we can define the
following map:
\begin{equation*}
\fiso_G (F, w) \longrightarrow \prod_{\bar{s} \in
\bar{\mathbf{s}}} ~ \fiso_G (F_{G \cdot \bar{s}},
w), \quad p \longmapsto \prod_{\bar{s} \in
\bar{\mathbf{s}}} ~ ~ p|_{F_{G \cdot \bar{s}}}.
\end{equation*}
It is easy that the map is homeomorphic. To
obtain a proof for the first statement, we will
show that the following map is a homeomorphism:
\begin{equation}
\tag{*} \fiso_G (F_{G \cdot \bar{s}}, w)
\longrightarrow \iso_{G_{\bar{s}}} (
F_{{\bar{s}}}, w ), \quad q \mapsto
q|_{F_{{\bar{s}}}}
\end{equation}
for each $\bar{s} \in \bar{\mathbf{s}}.$ For
this, we would decompose the map (*) into a
composition of two homeomorphisms. For $\bar{s}
\in \bar{\mathbf{s}},$ the equivariant vector
bundle $F_{G \cdot \bar{s}}$ is $G$-isomorphic to
$G \times_{G_{\bar{s}}} F_{\bar{s}}$ through the
inverse of the following $G$-isomorphism:
\[
G \times_{G_{\bar{s}}} F_{\bar{s}} \rightarrow
F_{G \cdot \bar{s}}, \quad [g, \bar{u}] \mapsto g
\bar{u},
\]
see \cite[Proposition II.3.2]{B}. By this, we
obtain a homeomorphism from $\fiso_G (F_{G \cdot
\bar{s}}, w)$ to $\fiso_G (G \times_{G_{\bar{s}}}
F_{\bar{s}}, w).$ Next, we consider the continous
map
\begin{equation}
\tag{**} \fiso_G (G \times_{G_{\bar{s}}}
F_{\bar{s}}, w) \rightarrow \iso_{G_{\bar{s}}} (
F_{{\bar{s}}}, w ), \quad q \mapsto
q|_{F_{\bar{s}}},
\end{equation}
where the fiber $[ \id, F_{\bar{s}} ]$ of $G
\times_{G_{\bar{s}}} F_{\bar{s}}$ is identified
with $F_{\bar{s}}.$ To show that the map is
homeomorphic, we construct its inverse. For any
$G_{\bar{s}}$-isomorphism $r$ from $F_{\bar{s}}$
to $w,$ define the following equivariant
fiberwise isomorphism:
\begin{equation*}
G \times_{G_{\bar{s}}} F_{\bar{s}}
\longrightarrow w, \quad [g, \bar{u}] \mapsto g
\cdot r ( \bar{u}).
\end{equation*}
In this way, we obtain a map from
$\iso_{G_{\bar{s}}} ( F_{{\bar{s}}}, w )$ to
$\fiso_G (G \times_{G_{\bar{s}}} F_{\bar{s}},
w).$ Easily, the map is the continuous inverse of
the map (**). So, we have obtained the following
decomposition of (*) into a composition of two
homeomorphisms:
\[
\fiso_G (F_{G \cdot \bar{s}}, w) ~
\longrightarrow ~ \fiso_G (G \times_{G_{\bar{s}}}
F_{\bar{s}}, w) ~ \longrightarrow ~
\iso_{G_{\bar{s}}} ( F_{{\bar{s}}}, w ).
\]
As a result, $\fiso_G (F_{G \cdot \bar{s}}, w)$
is homeomorphic to $\iso_{G_{\bar{s}}} (
F_{{\bar{s}}}, w ).$ Therefore, we obtain a proof
for the first statement.

Now, we prove the second statement. Since we have
already obtained sufficiency by Lemma \ref{lemma:
become extension}, we only have to show necessity. By
definition of $\ext F,$ the space $\iso_{G_{\bar{s}}}
( F_{{\bar{s}}}, w )$ for each $\bar{s} \in
\bar{\mathbf{s}}$ is nonempty. So, $\fiso_G (F, w)$ is
nonempty by the homeomorphism of the first statement,
and necessity holds. Therefore, we obtain a proof for
the second statement.
\end{proof}

%

\begin{corollary} \label{corollary:
path-connected} The set $\fiso_G (F, w)$ is
nonempty and path-connected for any $w \in \ext
F.$
\end{corollary}

\begin{proof}
By Proposition \ref{proposition: homeo type of
fiso}, we only have to show that
$\iso_{G_{\bar{s}}} ( F_{{\bar{s}}}, w )$ is
nonempty and path-connected for each $\bar{s} \in
\bar{\mathbf{s}}.$ By definition of $\ext F,$ the
representation $F_{\bar{s}}$ is
$G_{\bar{s}}$-isomorphic to $\res_{G_{\bar{s}}}^G
w$ for each $\bar{s} \in \bar{\mathbf{s}}.$ So,
$\iso_{G_{\bar{s}}} ( F_{{\bar{s}}}, w )$ for
each $\bar{s} \in \bar{\mathbf{s}}$ is nonempty,
and is homeomorphic to $\iso_{G_{\bar{s}}} ( w
).$ By Schur's lemma, $\iso_{G_{\bar{s}}} ( w )$
is homeomorphically group isomorphic to a product
of general linear groups over $\C,$ see Section
\ref{section: fundamental group} for detail.
Therefore, it is path-connected, and we obtain a
proof. Here, note that we use the assumption that
the scalar field is $\C.$
\end{proof}

Similar to $\Psi_F(w),$ we hereafter consider
$\fiso_G (F, w)$ just for $w \in \ext F$ by Lemma
\ref{lemma: become extension}. Now, we start
relating $\fiso_G (F, w)$ to $\Psi_F (w).$ By
Lemma \ref{lemma: equivalent condition for psi},
an equivariant pointwise clutching map $\psi$
determines an equivariant fiberwise isomorphism
$p_\psi.$ In the next useful lemma, we prove a
kind of converse of this, i.e. an equivariant
fiberwise isomorphism determines an equivariant
pointwise clutching map.

\begin{lemma} \label{lemma: define clutching}
For any $p \in \fiso_G (F, w)$ for $w \in \ext
F,$ there exists a unique map $\psi$ in $\Psi_F
(w)$ such that the following diagram commutes for
some $G$-isomorphism from $F / \psi$ to $w:$
\begin{equation*}
\SelectTips{cm}{} \xymatrix{ F \ar[r]^-{p}
\ar[d]_-{p_\psi} & w \\
F \big/ \psi \ar@{->}[ur]_-{\cong} }.
\end{equation*}
And, the set $\Psi_F (w)$ is nonempty for any $w
\in \ext F.$
\end{lemma}

\begin{proof}
First, we construct the unique pointwise
clutching map $\psi$ with respect to $F$
satisfying the diagram for some inequivariant
isomorphism from $F/\psi$ to $w.$ If there exists
such a map $\psi,$ then it should satisfy the
following:
\begin{equation}
\tag{*} p_\psi ( \bar{u} ) = p_\psi (
\bar{u}^\prime ) ~ \Longrightarrow ~ p ( \bar{u}
) = p ( \bar{u}^\prime ) \quad \text{ for any }
\bar{u}, \bar{u}^\prime \in F
\end{equation}
by the diagram. When $\bar{u} \in F_{\bar{x}}$
and $\bar{u}^\prime \in F_{\bar{x}^\prime}$ for
some $\bar{x}, \bar{x}^\prime \in
\bar{\mathbf{x}},$ each side of (*) is rewritten
as follows:
\begin{align*}
p_\psi ( \bar{u} ) = p_\psi ( \bar{u}^\prime ) ~
& \Longleftrightarrow ~ \bar{u}^\prime = \psi(
\bar{x}, \bar{x}^\prime ) \bar{u} \quad \qquad
\text{
and} \\
p ( \bar{u} ) = p ( \bar{u}^\prime ) ~ &
\Longleftrightarrow ~ p|_{F_{\bar{x}}} ( \bar{u}
) = p|_{F_{\bar{x}^\prime}} ( \bar{u}^\prime ).
\end{align*}
Substituting these into (*), the formula (*) is
rewritten as
\[
p|_{F_{\bar{x}}} ( \bar{u} ) ~ = ~
p|_{F_{\bar{x}^\prime}} \Big( \psi( \bar{x},
\bar{x}^\prime ) \bar{u} \Big) \qquad \text{ for
any } \bar{u} \in F_{\bar{x}} \text{ and }
\bar{x}^\prime \in \bar{\mathbf{x}}.
\]
So, we obtain
\begin{equation}
\tag{**} \psi(\bar{x}, \bar{x}^\prime) = (
p|_{F_{\bar{x}^\prime}} )^{-1} \circ
(p|_{F_{\bar{x}}}) \qquad \text{ for } \bar{x},
\bar{x}^\prime \in \bar{\mathbf{x}}
\end{equation}
because $p$ is a fiberwise isomorphism. The
unique bijective map $\alpha: F/\psi \rightarrow
w$ satisfying the diagram is expressed as
follows:
\[
\alpha(u) := p(\bar{u}) \quad \text{for any } u
\in F/\psi \text{ and } \bar{u} \in F \text{ such
that } p_\psi (\bar{u}) =u.
\]
This is because
\[
p (\bar{u}) ~=~ ( \alpha \circ p_\psi ) (\bar{u})
~=~ \alpha \big( p_\psi (\bar{u}) \big) ~=~
\alpha (u).
\]
It is easy that $\alpha$ is linear, so $\alpha$
is an inequivariant isomorphism satisfying the
diagram. And, $\psi$ is the wanted unique
pointwise clutching map.

Now, we will show that $\psi$ and $\alpha$ are
equivariant. Equivariance of $p$ can be expressed
as follows:
\[
p ( g^{-1} \bar{u} ) = g^{-1} \cdot p ( \bar{u} )
\qquad \text{ for each } g \in G \text{ and }
\bar{u} \in F.
\]
If we restrict the expression on each fiber
$F_{\bar{x}}$ for $\bar{x} \in \bar{\mathbf{x}},$ we
obtain the following:
\begin{equation}
\tag{***} \big( p|_{F_{g^{-1} \bar{x}}} \big)
g^{-1} = g^{-1} \cdot \big( p|_{F_{\bar{x}}}
\big),
\end{equation}
and by taking inverses of both terms of (***), we
obtain
\begin{equation}
\tag{****} g \big( p|_{F_{g^{-1} \bar{x}}}
\big)^{-1} = \big( p|_{F_{\bar{x}}} \big)^{-1} g
\qquad \text{ on } w.
\end{equation}
Then, we obtain
\begin{alignat*}{2}
(g \cdot \psi)(\bar{x}, \bar{x}^\prime) & ~=~ g
\psi( g^{-1} \bar{x}, g^{-1} \bar{x}^\prime
) g^{-1} & & \qquad \text{by definition of }
g \cdot \psi  \\
& ~=~ g ( p|_{F_{g^{-1} \bar{x}^\prime}} )^{-1}
\circ (p|_{F_{g^{-1} \bar{x}}}) g^{-1}
& & \qquad \text{by } (**) \\
& ~=~ \big[ \big( p|_{F_{\bar{x}^\prime}}
\big)^{-1} g \big] \circ \big[ g^{-1} \big(
p|_{F_{\bar{x}}} \big) \big]
& & \qquad \text{by } (***) \text{ and } (****) \\
& ~=~ \big( p|_{F_{\bar{x}^\prime}}
\big)^{-1} \circ \big( p|_{F_{\bar{x}}} \big) & & \\
& ~=~ \psi(\bar{x}, \bar{x}^\prime) & & \qquad
\text{by } (**)
\end{alignat*}
for any $g \in G$ and $\bar{x}, \bar{x}^\prime
\in \bar{\mathbf{x}},$ so $\psi$ is equivariant.
And, this implies that $p_\psi$ is equivariant by
Lemma \ref{lemma: equivalent condition for psi}.
Since the diagram commutes and two maps $p,$
$p_\psi$ are equivariant, the isomorphism
$\alpha$ is also equivariant. Therefore, we
obtain a proof for the first statement.

Now, we prove the second statement. For $w \in
\ext F,$ there exists an element $p$ in $\fiso
(F, w)$ by the second statement of Proposition
\ref{proposition: homeo type of fiso}. And,
existence of $p$ guarantees nonemptiness of
$\Psi_F (w)$ by the first statement of the lemma.
Therefore, we obtain a proof.
\end{proof}

By using Lemma \ref{lemma: define clutching}, we
can define a map $\Cl : \fiso_G (F, w)
\rightarrow \Psi_F (w)$ for $w \in \ext F$
sending an equivariant fiberwise isomorphism $p$
to the equivariant pointwise clutching map $\psi$
determined by $p,$ where $\Cl$ is the first two
letters of clutching map. It is easy that the map
$\Cl$ is continuous by (**) in the proof of Lemma
\ref{lemma: define clutching}. Though the map
$\Cl$ relates $\fiso_G (F, w)$ to $\Psi_F (w),$
readers would not be satisfied with it. To obtain
a more refined relation, we consider a continuous
$\iso_G (w)$-action on $\fiso_G (F, w)$ as
follows:
\[
\alpha \cdot p ~:=~ \alpha \circ p \qquad \text{
for } \alpha \in \iso_G (w) \text{ and } p \in
\fiso_G (F, w).
\]
Then, the map
\begin{equation*}
\cl : \fiso_G (F, w) / \iso_G (w) \longrightarrow
\Psi_F (w), \quad [p] \rightarrow \Cl (p)
\end{equation*}
is well-defined by definition of $\Cl,$ and also
continuous by the universal property of the
quotient map from $\fiso_G (F, w)$ to $\fiso_G
(F, w) / \iso_G (w).$ Moreover, the following
holds:

\begin{lemma} \label{lemma: cl is bijective}
For each $w \in \ext F,$ the map $\cl$ is
bijective.
\end{lemma}

\begin{proof}
First, we prove surjectivity. We only have to
show that $\Cl$ is surjective. Recalling that the
domain and codomain of $\cl$ are nonempty by
Proposition \ref{proposition: homeo type of fiso}
and Lemma \ref{lemma: define clutching}, we will
find a preimage of an arbitrary element $\psi \in
\Psi_F (w)$ under $\Cl.$ Note that $\alpha \circ
p_\psi$ for any $\alpha \in \iso_G ( F / \psi, w
)$ is an equivariant fiberwise isomorphism in
$\fiso_G (F, w),$ where nonemptiness of $\iso_G (
F / \psi, w )$ is guaranteed by definition of
$\Psi_F (w).$ By definition of $\Cl,$ it is easy
that $\Cl( \alpha \circ p_\psi ) = \psi.$ So,
$\Cl$ is surjective. Therefore, $\cl$ is
surjective.

Next, we prove injectivity. Assume that $\cl ([p]) =
\cl ([p^\prime])$ for $p,$ $p^\prime \in \fiso_G (F,
w).$ Put $\psi = \Cl (p) = \Cl (p^\prime).$ By
definition of $\Cl,$
\[
\psi(\bar{x}, \bar{x}^\prime) = (
p|_{F_{\bar{x}^\prime}} )^{-1} \circ
p|_{F_{\bar{x}}} = (
p^\prime|_{F_{\bar{x}^\prime}} )^{-1} \circ
p^\prime|_{F_{\bar{x}}}
\]
for any $\bar{x},$ $\bar{x}^\prime \in
\bar{\mathbf{x}}.$ Then, we have
\[
p^\prime|_{F_{\bar{x}^\prime}} \circ (
p|_{F_{\bar{x}^\prime}} )^{-1} =
p^\prime|_{F_{\bar{x}}} \circ
(p|_{F_{\bar{x}}})^{-1},
\]
i.e. $(p^\prime|_{F_{\bar{x}}}) \circ
(p|_{F_{\bar{x}}})^{-1}$ is constant regardless
of $\bar{x},$ and is contained in $\iso (w)$
because $p, p^\prime$ are fiberwise isomorphisms.
If we denote by $\beta$ this inequivariant
isomorphism, then we have $p^\prime = \beta \circ
p.$ Easily, $\beta \in \iso_G (w)$ because $p,
p^\prime$ are equivariant. So, $p^\prime = \beta
\cdot p,$ and hence $\cl ([p]) = \cl
([p^\prime]).$ Therefore, we obtain a proof.
\end{proof}

\begin{remark} \label{remark: inverse of cl}
In the proof of Lemma \ref{lemma: cl is
bijective}, we constructed the inverse $\cl^{-1}$
as follows:
\begin{equation*} \label{equation: inverse}
\cl^{-1} : \Psi_F (w) \longrightarrow \fiso_G (F,
w) / \iso_G (w), \qquad \psi \rightarrow [\alpha
\circ p_\psi]
\end{equation*}
for any map $\alpha \in \iso_G ( F / \psi, w ).$
Note that $\alpha$ is dependent of $\psi$ because
the domain $F / \psi$ of $\alpha$ is dependent of
$\psi.$ To stress this dependence, we will denote
$\alpha$ by $\alpha_\psi.$ Since continuity of
the inverse is our next issue, we will refer to
the remark several times.
\end{remark}

For a time being, we will prove that $\cl^{-1}$
is also continuous for each $w \in \ext F.$ For
this, we should construct $p_\psi$ and
$\alpha_\psi$ which vary continuously with $\psi
\in \Psi_F$ by Remark \ref{remark: inverse of
cl}. At the present stage, this is not easy
because the representation space (i.e. the vector
space structure) of $F / \psi$ varies with
$\psi.$ To solve this problem, we need give a new
definition of $F / \psi.$ As a candidate of it,
we define a new $G$-representation. Pick a point
$\bar{x}_0 \in \bar{\mathbf{x}}.$

\begin{definition} \label{definition: star
operation} For $\psi \in \Psi_F,$ we define the
operation $\star_\psi$ of $G$ on $F_{\bar{x}_0}$
as follows:
\begin{equation*}
g \star_\psi \bar{u} ~:=~ \psi(g \bar{x}_0, \bar{x}_0)
g \bar{u} \qquad \textit{ for } g \in G \text{ and }
\bar{u} \in F_{\bar{x}_0}.
\end{equation*}
In particular, the operation restricted to
$G_{\bar{x}_0}$ is identical with the
representation $F_{\bar{x}_0},$ i.e. $g
\star_\psi \bar{u} = g \bar{u}$ for $g \in
G_{\bar{x}_0}.$
\end{definition}

We show that the operation becomes a
$G$-representation.

\begin{lemma} \label{lemma: become a representation}
For any $\psi \in \Psi_F,$ the operation
$\star_\psi$ of $G$ on $F_{\bar{x}_0}$ is a
complex $G$-representation.
\end{lemma}

\begin{proof}
Since $\id \star_\psi ~ \bar{u} = \bar{u}$ for
any $\bar{u} \in F_{\bar{x}_0}$ by definition, we
only have to show that
\[
g \star_\psi ( h \star_\psi \bar{u} ) = ( gh )
\star_\psi \bar{u} \qquad \textit{ for any } g, h
\in G \text{ and } \bar{u} \in F_{\bar{x}_0}.
\]
By calculation, we have
\begin{alignat*}{2}
g \star_\psi ( h \star_\psi \bar{u} ) &= \psi(g
\bar{x}_0, \bar{x}_0) g \psi(h \bar{x}_0,
\bar{x}_0) h \bar{u}
& & \quad \text{by definition of } \star_\psi \\
&= \psi(g \bar{x}_0, \bar{x}_0) g \psi( g^{-1} g
h \bar{x}_0, g^{-1} g \bar{x}_0) g^{-1} g h
\bar{u}
& &   \\
&= \psi(g \bar{x}_0, \bar{x}_0) \psi( g h
\bar{x}_0,
g \bar{x}_0) g h \bar{u}
& & \quad \text{by equivariance of } \psi \\
&= \psi(g h \bar{x}_0, \bar{x}_0) g h \bar{u}
& & \quad \text{by transitivity of } \psi \\
&= ( gh ) \star_\psi \bar{u}. & & \quad \text{by
definition of } \star_\psi
\end{alignat*}
\end{proof}

Note that the representation space of $(
F_{\bar{x}_0}, \star_\psi )$ does not change, and
that its $G$-representation structure varies
continuously with $\psi.$ In Lemma \ref{lemma:
precise expression}, we will see that $(
F_{\bar{x}_0}, \star_\psi )$ replaces $F/\psi.$
If definition of $F/\psi$ is changed, definition
of $p_\psi$ should be also changed so as to be in
accordance with $( F_{\bar{x}_0}, \star_\psi ).$
The following $p_\psi^\prime$ is a candidate for
a new definition of $p_\psi.$

\begin{lemma} \label{lemma: become equivariant fiberwise
isomorphism}
For any $\psi \in \Psi_F,$ the map
\[
p_\psi^\prime : F \rightarrow ( F_{\bar{x}_0},
\star_\psi ), \quad \bar{u} \mapsto \psi(\bar{x},
\bar{x}_0) \bar{u} \qquad \text{ for } \bar{u}
\in F_{\bar{x}}
\]
is an equivariant fiberwise isomorphism.
\end{lemma}

\begin{proof}
Since it is easy to show that $p_\psi^\prime$ is
a fiberwise isomorphism, we only have to show
that $p_\psi^\prime$ is equivariant, i.e.
\[
p_\psi^\prime (g \bar{u}) = g \star_\psi
p_\psi^\prime (\bar{u})
\]
for any $g \in G$ and $\bar{u} \in F_{\bar{x}}.$
By calculation, we have
\begin{alignat*}{2}
p_\psi^\prime (g \bar{u}) &~=~ \psi(g \bar{x},
\bar{x}_0) g \bar{u}
& & \qquad \text{by definition of } p_\psi^\prime \\
&~=~ \psi( g \bar{x}_0, \bar{x}_0) \psi( g
\bar{x}, g \bar{x}_0) g \bar{u}
& & \qquad \text{by transitivity of } \psi \\
&~=~ \psi(g \bar{x}_0, \bar{x}_0) g \psi(g^{-1} g
\bar{x}, g^{-1} g \bar{x}_0) g^{-1} g \bar{u}
& & \qquad \text{by equivariance of } \psi \\
&~=~ \psi( g \bar{x}_0, \bar{x}_0) g \psi(
\bar{x}, \bar{x}_0) \bar{u}
& &  \\
&~=~ g \star_\psi \psi( \bar{x}, \bar{x}_0)
\bar{u}
& & \qquad \text{by definition of } \star_\psi \\
&~=~ g \star_\psi p_\psi^\prime ( \bar{u} ) & &
\qquad \text{by definition of } p_\psi^\prime
\end{alignat*}
for $g \in G$ and $\bar{u} \in F_{\bar{x}}.$
Therefore, we obtain a proof.
\end{proof}

Easily, the codomain $F_{\bar{x}_0}$ of
$p_\psi^\prime$ does not change, and the map
$p_\psi^\prime$ varies continuously with $\psi.$
Now, we can replace $F/\psi$ and $p_\psi$ with $(
F_{\bar{x}_0}, \star_\psi )$ and $p_\psi^\prime,$
respectively.

\begin{lemma} \label{lemma: precise expression}
For any $\psi \in \Psi_F,$ the following diagram
commutes for some $G$-isomorphism from $F/\psi$
to $( F_{\bar{x}_0}, \star_\psi )$:
\begin{equation*}
\SelectTips{cm}{} \xymatrix{ F
\ar[r]^-{p_\psi^\prime} \ar[d]_-{p_\psi} & (
F_{\bar{x}_0}, \star_\psi ) \\
F \big/ \psi \ar@{->}[ur]_-{\cong} }.
\end{equation*}
\end{lemma}

\begin{proof}
By Lemma \ref{lemma: become equivariant fiberwise
isomorphism}, $p_\psi^\prime$ is an equivariant
fiberwise isomorphism from $F$ to the
$G$-representation $( F_{\bar{x}_0}, \star_\psi
).$ By definition of $\Cl,$ we can check
$\Cl(p_\psi^\prime)=\psi.$ By Lemma \ref{lemma:
define clutching}, we obtain the diagram.
\end{proof}

By Lemma \ref{lemma: precise expression}, we can
henceforward put
\[
F/\psi := ( F_{\bar{x}_0}, \star_\psi ) \qquad
\text{ and } \qquad p_\psi := p_\psi^\prime.
\]
We can observe that previous results stated by
using previous $F/\psi$ and $p_\psi$ still hold
for new $F/\psi$ and $p_\psi.$ Now, we return to
our arguments below Remark \ref{remark: inverse
of cl} to show that $\cl^{-1}$ is continuous.
Since we constructed $p_\psi$ in Lemma
\ref{lemma: become equivariant fiberwise
isomorphism} which varies continuously with $\psi
\in \Psi_F (w),$ the remaining is to construct
$\alpha_\psi$ in $\iso_G ( F / \psi, w )$ which
varies continuously with $\psi \in \Psi_F (w).$

\begin{lemma} \label{lemma: averaging}
For arbitrary $\psi_0$ in $\Psi_F,$ put $w =
F/\psi_0.$ Then, there exists a continuous map
\[
\boldsymbol{\alpha} : U \subset \Psi_F(w)
\rightarrow \iso (F_{\bar{x}_0})
\]
for a small neighborhood $U$ of $\psi_0$ such
that $\boldsymbol{\alpha} (\psi)$ is contained in
$\iso_G \big( F/\psi, w \big)$ for each $\psi \in
U.$
\end{lemma}

\begin{proof}
Note that $w \in \ext F$ and $\psi_0 \in \Psi_F
(w)$ by definition of $\gl.$ For $\psi$ in
$\Psi_F (w),$ let $\rho_\psi : G \rightarrow
\iso( F_{\bar{x}_0} )$ be the group isomorphism
corresponding to the $G$-representation $F/\psi =
( F_{\bar{x}_0}, \star_\psi ),$ i.e. $\rho_\psi
(g) \bar{u} = g \star_\psi \bar{u}$ for $g \in G$
and $\bar{u} \in F_{\bar{x}_0}.$ Note that
$\rho_\psi$ varies continuously with $\psi$
because the $G$-representation structure of
$F/\psi$ varies continuously with $\psi.$ By
using $\rho_\psi,$ we define a map
$\boldsymbol{\alpha}$ as follows:
\[
\boldsymbol{\alpha} : \Psi_F(w) \rightarrow \End
(F_{\bar{x}_0}), \quad \psi \mapsto \int_G
\rho_{\psi_0} (g) \rho_\psi (g)^{-1} ~ dg,
\]
where $\End ( \cdot )$ is the set of
inequivariant endomorphisms of a vector space and
$dg$ is a Haar measure. Then, it is trivial that
$\boldsymbol{\alpha} (\psi_0)= \mathrm{constant}
\cdot \id_{F_{\bar{x}_0}}$ by definition. And, it
is easy that $\boldsymbol{\alpha}$ is continuous
because $\rho_\psi$ varies continuously with
$\psi.$ So, the remaining is to show that
$\boldsymbol{\alpha} (\psi)$ is contained in
$\iso_G \big( F/\psi, w \big)$ for each $\psi$ in
a small neighborhood $U$ of $\psi_0.$ First, we
show that $\boldsymbol{\alpha} (\psi)$ is
contained in $\End_G \big( F/\psi, w \big)$ for
each $\psi \in \Psi_F (w),$ i.e.
\[
\rho_{\psi_0} (g^\prime) \boldsymbol{\alpha}
(\psi) \rho_\psi (g^\prime)^{-1} =
\boldsymbol{\alpha} (\psi)
\]
for each $\psi \in \Psi_F (w)$ and $g^\prime \in
G$ because group isomorphisms corresponding to
$F/\psi$ and $w$ are $\rho_{\psi}$ and
$\rho_{\psi_0},$ respectively. By calculation, we
have
\begin{alignat*}{2}
\rho_{\psi_0} (g^\prime) \boldsymbol{\alpha}
(\psi) \rho_\psi (g^\prime)^{-1} &~=~ \int_G
\rho_{\psi_0} (g^\prime) \rho_{\psi_0} (g)
\rho_\psi (g)^{-1} \rho_\psi (g^\prime)^{-1} ~ dg
\\
&~=~ \int_G \rho_{\psi_0} (g^\prime g) \rho_\psi
(g^\prime g)^{-1} ~ dg
\\
&~=~ \int_G \rho_{\psi_0} (g) \rho_\psi (g)^{-1} ~ dg
\quad \text{by invariance of Haar measure} \\
&~=~ \boldsymbol{\alpha} (\psi).
\end{alignat*}
So, we have shown that $\boldsymbol{\alpha}
(\psi)$ is contained in $\End_G \big( F/\psi, w
\big)$ for each $\psi \in \Psi_F (w).$
Furthermore since $\boldsymbol{\alpha} (\psi_0)=
\mathrm{constant} \cdot \id_{F_{\bar{x}_0}},$ the
endomorphism $\boldsymbol{\alpha} (\psi)$ is an
isomorphism for each $\psi$ in a small
neighborhood $U$ of $\psi_0.$ So, the restriction
of $\boldsymbol{\alpha}$ to $U$ is a wanted map.
\end{proof}

We are ready to prove that $\cl : \fiso_G (F, w)
/ \iso_G (w) \longrightarrow \Psi_F (w)$ is a
homeomorphism for $w \in \ext F.$

\begin{proposition} \label{proposition: cl is homeomorphic}
For each $w \in \ext F,$ the map $\cl$ is
homeomorphic. And, $\Psi_F (w)$ is nonempty and
path-connected.
\end{proposition}

\begin{proof}
As we have seen before, $\cl$ is bijective and
continuous. So, we only have to show that
$\cl^{-1}$ is continuous. By Remark \ref{remark:
inverse of cl}, the inverse $\cl^{-1}$ is as
follows:
\begin{equation*}
\cl^{-1} : \Psi_F (w) \longrightarrow \fiso_G (F,
w) / \iso_G (w), \quad \psi \rightarrow
[\alpha_\psi \circ p_\psi]
\end{equation*}
for any $G$-isomorphism $\alpha_\psi$ in $\iso_G
( F / \psi, w ).$ To show that $\cl^{-1}$ is
continuous, we will show that $\cl^{-1}$ is
continuous at an arbitrary $\psi_0 \in \Psi_F
(w).$ Instead of $\alpha_\psi,$ we consider the
continuous map $\boldsymbol{\alpha}$ of Lemma
\ref{lemma: averaging} defined on a small
neighborhood $U$ of $\psi_0$ whose image
$\boldsymbol{\alpha} (\psi)$ is contained in
$\iso_G ( F / \psi, w )$ for each $\psi \in U.$
Then, the inverse $\cl^{-1}$ restricted to $U$ is
expressed as follows:
\begin{equation*}
\cl^{-1}|_U : U \longrightarrow \fiso_G (F, w) /
\iso_G (w), \quad \psi \rightarrow
[\boldsymbol{\alpha}(\psi) \circ p_\psi].
\end{equation*}
Since both $\boldsymbol{\alpha}(\psi)$ and
$p_\psi$ vary continuously with $\psi,$ the map
$\cl^{-1}$ on $U$ is continuous. In particular,
the map $\cl^{-1}$ is continuous at $\psi_0.$
Therefore, we obtain a proof for the first
statement.

Now, we prove the second statement. By Corollary
\ref{corollary: path-connected}, $\fiso_G (F, w)$
is nonempty and path-connected. So, the quotient
space $\fiso_G (F, w) / \iso_G (w)$ is nonempty
and path-connected. Since the map $\cl$ is
homeomorphic by the first statement, we obtain
that $\Psi_F (w)$ is nonempty and path-connected.
Therefore, we obtain a proof.
\end{proof}

Now, we can prove Theorem \ref{theorem:
bijectivity with extensions} and Theorem
\ref{theorem: homeo type of Psi}.

\begin{proof}[Proof of Theorem
\ref{theorem: bijectivity with extensions}] By
arguments below Lemma \ref{lemma: become a
representation}, the representation space of
$F/\psi$ is constant regardless of $\psi,$ and
the $G$-representation structure of $F/\psi$
varies continuously with $\psi.$ So by definition
of $\gl,$ the map $\gl$ is continuous. By
Proposition \ref{proposition: cl is
homeomorphic}, each preimage of $\gl$ is
nonempty, i.e. $\gl$ is surjective. Again by
Proposition \ref{proposition: cl is
homeomorphic}, each preimage of $\gl$ is
path-connected, i.e. the map $[\gl]$ is
injective. Therefore, we obtain a proof.
\end{proof}

\begin{proof}[Proof of Theorem \ref{theorem: homeo type of Psi}]
Let $\bar{\mathbf{s}} \subset \bar{\mathbf{x}}$
be a subset containing exactly one element in
each orbit of $\bar{\mathbf{x}}.$ By Proposition
\ref{proposition: homeo type of fiso}, we have
the following homeomorphism:
\[
\imath : \fiso_G (F, w) ~ \longrightarrow ~
\prod_{{\bar{s}} \in \bar{\mathbf{s}}} ~
\iso_{G_{\bar{s}}} ( F_{{\bar{s}}}, w ), \quad p
~ \longmapsto ~ \prod_{{\bar{s}} \in
\bar{\mathbf{s}}} ~ p|_{F_{{\bar{s}}}}.
\]
If we define an $\iso_G (w)$-action on the
product $\prod_{{\bar{s}} \in \bar{\mathbf{s}}} ~
\iso_{G_{\bar{s}}} ( F_{{\bar{s}}}, w )$ as
follows:
\begin{equation*}
\alpha \cdot \prod_{\bar{s} \in \bar{\mathbf{s}}}
~ p_{\bar{s}} := \prod_{\bar{s} \in
\bar{\mathbf{s}}} ~ \alpha \circ p_{\bar{s}}
\end{equation*}
for $\alpha \in \iso_G (w)$ and $p_{\bar{s}} \in
\iso_{G_{\bar{s}}} ( F_{{\bar{s}}}, w ),$ then
$\imath$ is $\iso_G (w)$-homeomorphic. So,
$\Psi_F (w)$ is homeomorphic to
\[
\Bigg( \prod_{{\bar{s}} \in \bar{\mathbf{s}}} ~
\iso_{G_{\bar{s}}} ( F_{{\bar{s}}}, w ) \Bigg)
\bigg/ \iso_G (w)
\]
because $\Psi_F (w)$ is homeomorphic to $\fiso_G
(F, w) / \iso_G (w)$ by Proposition
\ref{proposition: cl is homeomorphic}. Since $w
\in \ext F$ and hence we have
$\res_{G_{\bar{s}}}^G w \cong F_{{\bar{s}}}$ by
definition, the topological space
$\iso_{G_{\bar{s}}} ( F_{{\bar{s}}}, w )$ is
$\iso_G ( w )$-homeomorphic to
$\iso_{G_{\bar{s}}} ( w ).$ Therefore, $\Psi_F
(w)$ is homeomorphic to
\[
\Bigg( \prod_{{\bar{s}} \in \bar{\mathbf{s}}} ~
\iso_{G_{\bar{s}}} ( w ) \Bigg) \bigg/ \iso_G
(w),
\]
and we obtain a proof.
\end{proof}

Applying Theorem \ref{theorem: homeo type of Psi}
to two extreme cases, we obtain the following
corollary:
\begin{corollary} \label{corollary: two extreme
cases} For each $w \in \ext F,$ the set $\Psi_F
(w)$ is homeomorphic to
\[
\left\{
\begin{array}{ll}
\iso_{G_{{\bar{x}}}} ( w ) \big/ \iso_G ( w )
\text{ for any } \bar{x} \in \bar{\mathbf{x}} &
\text{ if the $G$-action on
$\bar{\mathbf{x}}$ is transitive,} \\
\iso_G (w)^{|\bar{\mathbf{x}}|-1} & \text{ if the
$G$-action on $\bar{\mathbf{x}}$ is trivial.}
\end{array}
\right.
\]
\end{corollary}

\begin{proof}
The transitive case is trivial by Theorem
\ref{theorem: homeo type of Psi}.

Assume that the $G$-action on $\bar{\mathbf{x}}$
is trivial. Denote $|\bar{\mathbf{x}}|$ by $N.$
By Theorem \ref{theorem: homeo type of Psi},
$\Psi_F (w)$ is homeomorphic to $\iso_G (w)^N /
\iso_G (w).$ We consider the map
\[
j : \iso_G (w)^{N-1} \rightarrow \iso_G (w)^N /
\iso_G (w), \quad (\alpha_1, \cdots,
\alpha_{N-1}) \mapsto [(\alpha_1, \cdots,
\alpha_{N-1}, \id)].
\]
It is easy that $j$ is bijective. And if we consider
smooth structures of two spaces, it is easy that $j$
is locally diffeomorphic. Therefore, $j$ is a
homeomorphism.
\end{proof}

\bigskip

\section{The Fundamental group of the set
of equivariant pointwise clutching maps}
\label{section: fundamental group}

In this section, we investigate the fundamental
group of $\Psi_F (w)$ for $w \in \ext F$ when the
$G$-action on $\bar{\mathbf{x}}$ is transitive,
and prove Proposition \ref{proposition: simply
connected} and Corollary \ref{corollary: simply
connected}. When the $G$-action on
$\bar{\mathbf{x}}$ is transitive, $\Psi_F (w)$ is
homeomorphic to $\iso_{G_{{\bar{x}}}} ( w ) \big/
\iso_G ( w )$ for any $\bar{x} \in
\bar{\mathbf{x}}$ by Corollary \ref{corollary:
two extreme cases}. Hence to calculate $\pi_1
\big( \Psi_F (w) \big),$ we need understand the
following long exact sequence of homotopy groups
for the fibration $\iso_{G_{\bar{x}}} (w)
\rightarrow \iso_{G_{\bar{x}}} (w) / \iso_G (w)$
with the fiber $\iso_G (w)$ :
\begin{align} \label{equation: exact sequence of
fibration} & \pi_1 \Big( \iso_G (w) \Big)
\overset{\pi_1(k)}{\longrightarrow} \pi_1 \Big(
\iso_{G_{\bar{x}}} (w) \Big) \longrightarrow
\pi_1 \Big( \iso_{G_{\bar{x}}} (w) /
\iso_G (w) \Big) \\
\notag & \longrightarrow \pi_0 \Big( \iso_G (w)
\Big) \longrightarrow \pi_0 \Big(
\iso_{G_{\bar{x}}} (w) \Big),
\end{align}
where $\pi_1(k)$ is a map induced by the
inclusion $k: \iso_G (w) \rightarrow
\iso_{G_{\bar{x}}} (w).$ Since $\pi_0$'s in the
long exact sequence are trivial, we only have to
understand $\pi_1 (k)$ to calculate $\pi_1 \big(
\iso_{G_{\bar{x}}} (w) / \iso_G (w) \big),$ see
arguments below Example \ref{example: group
isomorphism} for triviality of $\pi_0$'s. For
notational simplicity, we generalize the
situation slightly. Instead of $\pi_1(k)$ in
(\ref{equation: exact sequence of fibration}), we
will investigate
\begin{equation} \label{equation: induced map}
\pi_1 \Big( \iso_G (w) \Big) \longrightarrow \pi_1
\Big( \iso_K (w) \Big)
\end{equation}
for a complex $G$-representation $w$ and a closed
subgroup $K \subset G,$ which is a map induced by
the inclusion $\iso_G (w) \hookrightarrow \iso_K
(w).$ So in this section, we often refer to the
map (\ref{equation: induced map}).

Let $w$ be a complex $G$-representation, and $K$ be a
closed subgroup of $G.$ The $G$-representation $w$ is
decomposed into a direct sum of irreducible
$G$-representations as follows:
\begin{equation} \label{equation: irr
decomposition}
w \cong l_1 \cdot U_1 \oplus \cdots \oplus l_m
\cdot U_m
\end{equation}
for some natural numbers $l_i$'s and irreducible
$G$-representations $U_i$'s such that $U_i
\not\cong U_{i^\prime}$ whenever $i \ne
i^\prime.$ Then by Schur's lemma, we have a
homeomorphic group isomorphism
\begin{equation} \label{equation: group iso for
G} \iso_G ( w ) \cong \GL(l_1, \C) \times \cdots
\times \GL(l_m, \C),
\end{equation}
see \cite[p. 69]{BD} for Schur's lemma and see
\cite[Exercise 9 of p. 72]{BD} for the group
isomorphism. For explanation of the group
isomorphism (\ref{equation: group iso for G}), we
give an example.

\begin{example} \label{example: group isomorphism}
Assume that $m=2$ and $l_1=2,$ $l_2=3$ for the
decomposition (\ref{equation: irr
decomposition}), i.e.
\[
w \cong ( U_1 \oplus U_1 ) \oplus ( U_2 \oplus
U_2 \oplus U_2 ).
\]
If we express an element in $\GL(2, \C) \times
\GL(3, \C)$ of (\ref{equation: group iso for G})
as
\[
\left(
\begin{array}{c|c}
    A          & \mathbf{0}  \\
    \hline
    \mathbf{0} & B  \\
  \end{array}
\right)
\]
for $A \in \GL(2, \C)$ and $B \in \GL(3, \C),$
the element is corresponding to the following
element in $\iso_G(w)$ through the isomorphism
(\ref{equation: group iso for G}):
\[
\left( ~ \left(\begin{array}{c}
u_1  \\
u_2  \\
\end{array}\right), ~ \left(
\begin{array}{c}
v_1  \\
v_2  \\
v_3  \\
\end{array}
\right) ~ \right) ~ \longmapsto ~ \left( ~ A
\cdot \left(
\begin{array}{c}
u_1  \\
u_2  \\
\end{array}
\right) , ~ B \cdot \left(
\begin{array}{c}
v_1  \\
v_2  \\
v_3  \\
\end{array}
\right) ~ \right)
\]
for $\Big( ( u_1, u_2), ~ (v_1, v_2, v_3) \Big)
\in ( U_1 \oplus U_1 ) \oplus ( U_2 \oplus U_2
\oplus U_2 ),$ where $\mathbf{0}$ is a zero
matrix with a suitable size. \qed
\end{example}

Now, we digress briefly into the zeroth and first
homotopy groups of $\iso_K ( w )$ for a complex
$G$-representation $w$ and a closed subgroup $K
\subset G.$ First, the space $\iso_K ( w )$ is
homeomorphically group isomorphic to a product of
general linear groups over $\C$ in the exactly
same way with $\iso_G ( w ).$ It is well known
that general linear groups over $\C$ are
path-connected, and that their fundamental groups
are $\Z,$ more precisely the determinant map
$\det : \GL(l, \C) \rightarrow \C^*$ for $l \in
\N$ induces the group isomorphism
\[
\pi_1 (\det) : \pi_1 \big( \GL(l, \C) \big)
\rightarrow \pi_1 ( \C^* ) \cong \Z.
\]
So, the zeroth homotopy group of $\iso_K ( w )$
for any closed subgroup $K \subset G$ is trivial
as mentioned earlier. And, if $\iso_K ( w )$ is
homeomorphically group isomorphic to a product of
$m$ general linear groups over $\C,$ then the
fundamental group of $\iso_K ( w )$ is $\Z^m.$

We return to our investigation of the map
(\ref{equation: induced map}) induced by the
inclusion. In this time, we will apply Schur's
lemma to the restricted $K$-representation
$\res_K^G w$ for a closed subgroup $K$ of $G$ so
as to investigate $\iso_K ( w ).$ In
understanding the restricted representation, it
is helpful to investigate whether irreducible
$G$-representations $U_i$'s contained in $w$
satisfy the following condition on
irreducibility:
\begin{itemize}
\item[] Condition (I). $\res_K^G U_i$'s are all
irreducible.
\end{itemize}
For a while, we assume Condition (I). Under the
condition, our investigation of $\res_K^G w$ is
divided into two steps: first, we additionally
assume that
\[
\res_K^G U_i \text{'s are all isomorphic,}
\]
and then we get rid of the assumption later.

~

\subsection*{The first step: Condition (I) holds, and $\res_K^G U_i$'s are all
isomorphic.} ~

By the assumption of this step, $\res_K^G w$ is
decomposed into a direct sum of irreducible
$K$-representations as follows:
\[
\res_K^G w \cong (l_1+\cdots+l_m) \cdot \res_K^G U_1.
\]
By Schur's lemma, we have a homeomorphic group
isomorphism
\begin{equation} \label{equation: group iso for
subgroup} \iso_K ( w ) \cong \GL(l_1+\cdots+l_m,
\C) .
\end{equation}
And, the inclusion $\iso_G ( w ) \hookrightarrow
\iso_K ( w )$ is regarded as the injection
\begin{align} \label{equation: map between GL 1}
I_{l_1, \cdots, l_m} : \GL(l_1, \C) \times \cdots
\times \GL(l_m, \C)
& ~\longrightarrow~ \GL(l_1+\cdots+l_m, \C), \\
\notag (A_1, \cdots, A_m) & ~\longmapsto~ \left(
\begin{array}{c|c|c|c}
A_1         & \mathbf{0}  & \mathbf{0}  & \mathbf{0} \\
\hline
\mathbf{0}  & A_2         & \mathbf{0}  & \mathbf{0} \\
\hline
\mathbf{0}  & \mathbf{0}  & \ddots      & \mathbf{0} \\
\hline
\mathbf{0}  & \mathbf{0}  & \mathbf{0}  & A_m \\
\end{array}
\right)
\end{align}
through two group isomorphisms (\ref{equation:
group iso for G}), (\ref{equation: group iso for
subgroup}). That is, the following diagram
commutes:
\begin{equation} \label{diagram: for inclusion}
\begin{CD}
\iso_G ( w )  @>{\quad \mathrm{inclusion}
\quad}>>
\iso_K ( w )      \\
@V{\cong}VV
@VV{\cong}V   \\
\GL(l_1, \C) \times \cdots \times \GL(l_m, \C)
\quad @>{\quad \mathrm{I_{l_1, \cdots, l_m}}
\quad}>> \quad \GL(l_1+\cdots+l_m, \C).
\end{CD}
\end{equation}
By the diagram, we can prove the following lemma:

\begin{lemma} \label{lemma: first case}
For a complex $G$-representation $w$ and a closed
subgroup $K \subset G,$ assume that Condition (I)
holds, and that $\res_K^G U_i$'s are all
isomorphic. Then, the map (\ref{equation: induced
map})
\[
\pi_1 \Big( \iso_G (w) \Big) \longrightarrow
\pi_1 \Big( \iso_K (w) \Big)
\]
induced by the inclusion is surjective.
\end{lemma}

\begin{proof}
By the diagram (\ref{diagram: for inclusion}), we
only have to show that $\pi_1(I_{l_1, \cdots,
l_m})$ is surjective. We define the following
diagonal matrices:
\[
A_1(t) = \left(
\begin{array}{c|c}
e^{2 \pi t \imath}  & \mathbf{0}   \\
\hline
\mathbf{0}          & \Id_{l_1-1}   \\

\end{array}
\right) \in \GL(l_1, \C), \text{ } ~ A_2(t) =
\Id_{l_2}, \text{ } ~ \cdots  \text{ } ~ , \text{
} ~ A_m(t) = \Id_{l_m}
\]
for $t \in [0,1],$ where $\Id_l$ is the identity
matrix of size $l.$ Then, the loop
\[
I_{l_1, \cdots, l_m} \Big(A_1(t), ~ A_2(t), ~ \cdots,
~ A_m(t) \Big)
\]
for $t \in [0,1]$ becomes a generator in $\pi_1
\Big( \GL(l_1+\cdots+l_m, \C) \Big) \cong \Z.$
Therefore, we obtain a proof.
\end{proof}
~

\subsection*{The second step: Condition (I) holds, but
$\res_K^G U_i$'s are not necessarily isomorphic.} ~

We relabel the decomposition (\ref{equation: irr
decomposition}) of $w$ as follows:
\begin{equation} \label{equation: irr decomposition
refined} w \cong \bigoplus_{1 \le i \le a} \Bigg(
\bigoplus_{ 1 \le j_i \le b_i} l_{i,j_i} \cdot
U_{i,j_i} \Bigg)
\end{equation}
for some $a,$ $b_i,$ $l_{i, j_i} \in \N$ and
irreducible $G$-representations $U_{i,j_i}$ so
that
\begin{enumerate}
\item $i = i^\prime$ and $j_i = j_{i^\prime}^\prime ~$
if and only if $~ U_{i,j_i} \cong U_{i^\prime,
j_{i^\prime}^\prime},$
\item $i = i^\prime ~$ if and only if
$~ \res_K^G U_{i,j_i} \cong \res_K^G U_{i^\prime,
j_{i^\prime}^\prime}.$
\end{enumerate}
In the relabeled decomposition (\ref{equation:
irr decomposition refined}), if we put
\[
w_i := \bigoplus_{ 1 \le j_i \le b_i} l_{i,j_i} \cdot
U_{i,j_i},
\]
then $w \cong \oplus_i ~ w_i.$ And, $w_i$'s satisfy
the following:
\begin{itemize}
\item[(a)] for each $i,$ all irreducible representations contained in
$w_i$ satisfy the assumption of the first step,
\item[(b)] whenever $i \ne i^\prime,$ there is no irreducible $K$-representation
contained in both $\res_K^G w_i$ and $\res_K^G
w_{i^\prime}.$
\end{itemize}
By (a), we can apply the first step to each $w_i$
to obtain
\[
\iso_G ( w_i ) \cong \prod_{1 \le j_i \le b_i}
\GL(l_{i,j_i}, ~ \C) \quad \text{ and } \quad
\iso_K ( w_i ) \cong \GL \Bigg( \sum_{1 \le j_i
\le b_i} l_{i,j_i}, ~ \C \Bigg).
\]
Then by Schur's lemma and (b), we have
\begin{align}
\label{equation: group iso for G refined} \iso_G ( w )
& \cong \prod_{1 \le i \le a} \iso_G ( w_i ) \cong
\prod_{1 \le i \le a} \Bigg( \prod_{1 \le
j_i \le b_i} \GL(l_{i,j_i}, ~ \C) \Bigg) ~ \text{ and} \\
\label{equation: group iso for subgroup refined}
\iso_K ( w ) & \cong \prod_{1 \le i \le a} \iso_K
( w_i ) \cong \prod_{1 \le i \le a} \GL \Bigg(
\sum_{1 \le j_i \le b_i} l_{i,j_i}, ~ \C \Bigg).
\end{align}

Next, we would find out a matrix expression for
the inclusion $\iso_G ( w ) \hookrightarrow
\iso_K ( w )$ through (\ref{equation: group iso
for G refined}) and (\ref{equation: group iso for
subgroup refined}). Since (a) holds, the
inclusion $\iso_G ( w_i ) \hookrightarrow \iso_K
( w_i )$ is regarded as the injection
\begin{align*}
I_{l_{i,1}, \cdots, l_{i, b_i}} : \prod_{1 \le j_i \le
b_i} \GL(l_{i,j_i}, ~ \C)
& ~\longrightarrow~ \GL \Bigg( \sum_{1 \le j_i \le b_i} l_{i,j_i}, ~ \C \Bigg), \\
\notag \prod_{1 \le j_i \le b_i} A_{i,j_i} &
~\longmapsto~ \left(
\begin{array}{c|c|c|c}
A_{i,1}     & \mathbf{0}  & \mathbf{0}  & \mathbf{0} \\
\hline
\mathbf{0}  & A_{i,2}     & \mathbf{0}  & \mathbf{0} \\
\hline
\mathbf{0}  & \mathbf{0}  & \ddots      & \mathbf{0} \\
\hline
\mathbf{0}  & \mathbf{0}  & \mathbf{0}  & A_{i,b_i} \\
\end{array}
\right)
\end{align*}
by (\ref{diagram: for inclusion}). We can piece
together these injections so that the inclusion
$\iso_G ( w ) \hookrightarrow \iso_K ( w )$ is
regarded as the injection
\begin{align} \label{equation: map between GL 2}
\prod_{1 \le i \le a} \Bigg( \prod_{1 \le j_i \le b_i}
\GL(l_{i,j_i}, ~ \C) \Bigg) & \quad \hookrightarrow
\quad \prod_{1 \le i
\le a} \GL \Bigg( \sum_{1 \le j_i \le b_i} l_{i,j_i}, \C \Bigg), \\
\notag \prod_{1 \le i \le a} \Bigg( \prod_{1 \le
j_i \le b_i} A_{i,j_i} \Bigg) & \quad \mapsto
\quad \prod_{1 \le i \le a} I_{l_{i,1}, \cdots,
l_{i, b_i}} \big( A_{i,1}, \cdots, A_{i, b_i}
\big)
\end{align}
through two group isomorphisms (\ref{equation:
group iso for G refined}), (\ref{equation: group
iso for subgroup refined}). We give an example of
(\ref{equation: map between GL 2}).

\begin{example}
Assume that $a=2$ and $b_1=2,$ $b_2=3$ for the
relabeled decomoposition (\ref{equation: irr
decomposition refined}), i.e.
\[
w \cong \big( l_{1,1} \cdot U_{1,1} \oplus
l_{1,2} \cdot U_{1,2} \big) \oplus \big( l_{2,1}
\cdot U_{2,1} \oplus l_{2,2} \cdot U_{2,2} \oplus
l_{2,3} \cdot U_{2,3} \big).
\]
Then,
\begin{align*}
\iso_G ( w ) & \cong \GL(l_{1,1}, \C) \times
\GL(l_{1,2}, \C) \times \GL(l_{2,1}, \C) \times
\GL(l_{2,2}, \C) \times \GL(l_{2,3}, \C) ~ \text{ and} \\
\iso_K ( w ) & \cong \GL(l_{1,1}+l_{1,2}, \C)
\times \GL(l_{2,1}+l_{2,2}+l_{2,3}, \C)
\end{align*}
by Schur's lemma. And, the inclusion $\iso_G ( w
) \hookrightarrow \iso_K ( w )$ is regarded as
the injection
\begin{align*} \label{equation: map between GL 1}
\prod_{1 \le i \le a} \Bigg( \prod_{1 \le j_i \le
b_i} \GL(l_{i,j_i}, ~ \C) \Bigg) & ~
\hookrightarrow ~ \prod_{1 \le i
\le a} \GL \Bigg( \sum_{j_i} l_{i,j_i}, \C \Bigg), \\
\notag \Big( (A_{1,1}, A_{1,2}), (A_{2,1},
A_{2,2}, A_{2,3}) \Big) & ~ \mapsto ~ \left(
\left(
\begin{array}{c|c}
A_{1,1}     & \mathbf{0}   \\
\hline
\mathbf{0}  & A_{1,2}     \\
\end{array}
\right)_{,} \left(
\begin{array}{c|c|c}
A_{2,1}     & \mathbf{0}  & \mathbf{0}  \\
\hline
\mathbf{0}  & A_{2,2}     & \mathbf{0}  \\
\hline
\mathbf{0}  & \mathbf{0}  & A_{2,3}     \\
\end{array}
\right) \right)
\end{align*}
for $A_{i,j_i} \in \GL(l_{i,j_i}, \C).$ \qed
\end{example}

By the inclusion (\ref{equation: map between GL
2}), we can get rid of an assumption in Lemma
\ref{lemma: first case}.

\begin{lemma} \label{lemma: second case}
For a complex $G$-representation $w$ and a closed
subgroup $K$ of $G,$ assume that Condition (I)
holds. Then, the map (\ref{equation: induced
map})
\[
\pi_1 \Big( \iso_G (w) \Big) \longrightarrow
\pi_1 \Big( \iso_K (w) \Big)
\]
induced by the inclusion is surjective.
\end{lemma}

\begin{proof}
As in the proof of Lemma \ref{lemma: first case},
we only have to show that the $\pi_1$ map induced
by (\ref{equation: map between GL 2}) is
surjective. For this, it suffices to show that
$\pi_1(I_{l_{i,1}, \cdots, l_{i, b_i}})$ is
surjective for each $i.$ But, this holds by the
proof of Lemma \ref{lemma: first case}.
Therefore, we obtain a proof.
\end{proof}

Now, we can prove Proposition \ref{proposition:
simply connected} and Corollary \ref{corollary:
simply connected}.

\begin{proof}[Proof of Proposition \ref{proposition: simply connected}]
By Corollary \ref{corollary: two extreme cases},
$\Psi_F (w)$ is homeomorphic to the quotient
space $\iso_{G_{\bar{x}}} (w) / \iso_G (w)$ for
any $\bar{x} \in \bar{\mathbf{x}}.$ So, we only
have to show that $\iso_{G_{\bar{x}}} (w) /
\iso_G (w)$ is simply connected. For this, it
suffices to show that the map
\[
\pi_1 \Big( \iso_G (w) \Big)
\overset{\pi_1(k)}{\longrightarrow} \pi_1 \Big(
\iso_{G_{\bar{x}}} (w) \Big)
\]
in the long exact sequence (\ref{equation: exact
sequence of fibration}) of homotopy groups is
surjective. However, this holds by Lemma
\ref{lemma: second case} because Condition (I)
holds by the assumption of the proposition.
Therefore, we obtain a proof.
\end{proof}

\begin{proof}[Proof of Corollary \ref{corollary: simply connected}]
If $G$ is abelian, complex irreducible
$G$-representations are all one-dimensional. So, their
restrictions to a closed subgroup are still
irreducible. So, the Condition (I) holds. Therefore,
we obtain a proof by Proposition \ref{proposition:
simply connected}.
\end{proof}

If Condition (I) does not hold, then the map
(\ref{equation: induced map}) is not necessarily
surjective. Though we do not indulge into such a
situation any more in the paper, interested
readers might need the following lemma to
investigate the map:

\begin{lemma} \label{lemma: no condition I}
For a complex irreducible $G$-representation $w$
and a closed subgroup $K$ of $G,$ assume that
$\res_K^G w$ is decomposed into a direct sum of
irreducible representations as follows:
\[
l_1 \cdot V_1 \oplus \cdots \oplus l_m \cdot V_m
\]
for some $l_i \in \N$ and irreducible
$K$-representations $V_i$ such that $V_i \not
\cong V_{i^\prime}$ whenever $i \ne i^\prime.$
Then, the map (\ref{equation: induced map})
induced by the inclusion is as follows:
\[
\pi_1 \big( \iso_G ( w ) \big) \cong \Z ~
\longrightarrow ~ \pi_1 \big( \iso_K ( w ) \big)
\cong \Z^m, \quad n ~ \longmapsto ~ ( ~ l_1 n, ~
\cdots, ~ l_m n ~ )
\]
for $n \in \Z$ up to sign.
\end{lemma}

\begin{proof}
By Schur's lemma, we have
\[
\iso_G ( w ) \cong \GL(1, \C) \quad \text{ and }
\quad \iso_K ( w ) \cong \GL(l_1, \C) \times
\cdots \times \GL(l_m, \C).
\]
So, we obtain $\pi_1 \big( \iso_G ( w ) \big)
\cong \Z$ and $\pi_1 \big( \iso_K ( w ) \big)
\cong \Z^m.$ Since the loop $e^{2 \pi t \imath}
\in \GL(1, \C) = \C^*$ for $t \in [0,1]$ gives a
generator of $\pi_1 \big( \iso_G ( w ) \big),$ it
suffices to find its image to obtain a proof. The
loop considered as contained in $\iso_K ( w )$ is
expressed as
\[
\big( e^{2 \pi t \imath} \cdot \Id_{l_1}, \cdots,
e^{2 \pi t \imath} \cdot \Id_{l_m} \big) \in
\GL(l_1, \C) \times \cdots \times \GL(l_m, \C)
\]
for $t \in [0,1].$ Since the loop $e^{2 \pi t \imath}
\cdot \Id_l \in \GL(l, \C)$ for $t \in [0,1]$ and $l
\in \N$ gives $\pm l$ times a generator of $\pi_1
\big( \GL(l, \C) \big),$ we obtain a proof.
\end{proof}

%
%

\bigskip

\section{Lemmas on evaluation and restriction of an equivariant pointwise
clutching map} \label{section: lemmas}

To explain motivation for evaluation of an equivariant
pointwise clutching map by an example, we recall
Example \ref{example: tetrahedron}.

\begin{example}[Continued from Exercise
\ref{example: tetrahedron}] \label{example:
evaluation} For a vertex $x \in X,$ we consider
an arbitrary equivariant pointwise clutching map
$\psi$ with respect to $(\pi^* E)_{\pi^{-1}(x)}.$
Put $\bar{\mathbf{x}} := \pi^{-1}(x) = \{
\bar{x}_i | i \in \Z_3 \},$ see Figure
\ref{figure: evaluation}.
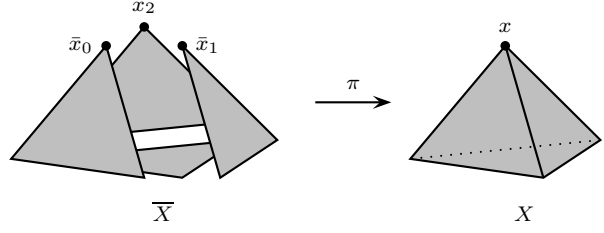
\begin{figure}[ht]
\begin{center}
\begin{pspicture}(-4,0)(5,3.5) \footnotesize

\pspolygon[fillstyle=solid,fillcolor=lightgray](1.5,
1)(3.25, 0.75)(4,1.25)(2.75,2.5)(1.5, 1)
\psline[linestyle=dotted](1.5, 1)(4,1.25)
\psline(3.25, 0.75)(2.75,2.5)

\pspolygon[fillstyle=solid,fillcolor=lightgray](-1.5,0.75)(-0.75,1.25)(-3.25,1)(-1.5,0.75)
\pspolygon[fillstyle=solid,fillcolor=lightgray](-0.75,1.5)(-2,2.75)(-3.25,1.25)(-0.75,1.5)
\pspolygon[fillstyle=solid,fillcolor=lightgray](-1,
0.75)(-0.25,1.25)(-1.5,2.5)(-1,0.75)
\pspolygon[fillstyle=solid,fillcolor=lightgray](-2,0.75)(-2.5,2.5)(-3.75,1)(-2,0.75)

\psline[arrowsize=5pt]{->}(0.25,1.75)(1.25,1.75)

\psdots(-2.5,2.5)(-2,2.75)(-1.5,2.5)(2.75,2.5)


\uput[u](0.75,1.75){$\pi$} \uput[u](3,0){$X$}
\uput[u](-1.75,0){$\overline{X}$}

\uput[l](-2.5,2.5){$\bar{x}_0$}
\uput[u](-2,2.75){$\bar{x}_2$}
\uput[r](-1.5,2.5){$\bar{x}_1$}
\uput[u](2.75,2.5){$x$}
\end{pspicture}
\end{center}
\caption{\label{figure: evaluation} The preimage
of a vertex}
\end{figure}
Then, the map $\psi$ is determined by the
evaluation $\psi(\bar{x}_0, \bar{x}_1),$ i.e. the
following are determined by $\psi(\bar{x}_0,
\bar{x}_1)$:
\[
\psi(\bar{x}_i, \bar{x}_i), ~ \psi(\bar{x}_i,
\bar{x}_{i+1}), ~ \psi(\bar{x}_{i+1}, \bar{x}_i)
\quad \text{ for } i \in \Z_3.
\]
We give details on this. The isotropy subgroup
$\tetra_x$ at $x$ of the tetrahedral group
$\tetra$ is the dihedral group of order 6. Let
$a$ and $b$ be elements of $\tetra_x$ such that
\[
a \cdot \bar{x}_i = \bar{x}_{i+1} \quad \text{
and } \quad b \cdot \bar{x}_i = \bar{x}_{-i}
\]
for $i \in \Z_3.$ By reflexivity of pointwise
clutching map, we have
\[
\psi (\bar{x}_i, \bar{x}_i) = \id \qquad
\text{for any } i \in \Z_3.
\]
By equivariance of $\psi,$ we have
\begin{equation*}
\psi(\bar{x}_i, \bar{x}_{i+1}) ~ = ~ a^i ~ \psi(
\bar{x}_0, \bar{x}_1 ) ~ a^{-i} \qquad \text{for
any } i \in \Z_3.
\end{equation*}
By symmetry of pointwise clutching map, we have
\begin{align*}
\psi(\bar{x}_{i+1}, \bar{x}_i) &= \psi(\bar{x}_i,
\bar{x}_{i+1})^{-1} \\
&= \big( a^i ~ \psi(
\bar{x}_0, \bar{x}_1 ) ~ a^{-i} \big)^{-1} \\
&= a^i ~ \psi( \bar{x}_0, \bar{x}_1 )^{-1} ~
a^{-i}
\end{align*}
for any $i \in \Z_3.$ So, we have shown that
$\psi$ is determined by its evaluation at
$(\bar{x}_0, \bar{x}_1).$ \qed
\end{example}

We will prove a result generalizing Example
\ref{example: evaluation}. Prior to this, we
introduce some terminologies. A subset $B \subset
\bar{\mathbf{x}} \times \bar{\mathbf{x}}$ is
called a \textit{binary relation} on
$\bar{\mathbf{x}}.$ For binary relations $B,
B^\prime$ on $\bar{\mathbf{x}},$ we define the
following three operations:
\begin{align*}
B^{-1} &= \Big\{ (\bar{x}^\prime, \bar{x}) ~|~
(\bar{x}, \bar{x}^\prime)
\in B \Big\},  \\
B^\prime \circ B &= \Big\{ (\bar{x},
\bar{x}^{\prime\prime}) ~|~ \text{ there exists }
\bar{x}^\prime \text{ such that } (\bar{x},
\bar{x}^\prime) \in B \text{ and } (\bar{x}^\prime,
\bar{x}^{\prime\prime}) \in B^\prime \Big\}, ~ \text{and} \\
G \cdot B &= \Big\{ (g \bar{x}, g \bar{x}^\prime)
~|~ g \in G \text{ and } (\bar{x},
\bar{x}^\prime) \in B \Big\}.
\end{align*}
A binary relation $B$ on $\bar{\mathbf{x}}$ is
called an \textit{equivariant equivalence
relation} if the following four hold:
\begin{enumerate}
\item[(a)] reflexivity : $~\Delta \subset B, $
\item[(b)] symmetry : $~B^{-1} \subset B,$
\item[(c)] transitivity : $~B \circ B \subset B,$
\item[(d)] equivariance : $~G \cdot B \subset B,$
\end{enumerate}
where $\Delta$ is the diagonal subset of
$\bar{\mathbf{x}} \times \bar{\mathbf{x}}.$ Then,
there exists the smallest equivariant equivalent
relation containing $B$ for any binary relation
$B$ on $\bar{\mathbf{x}}.$ More precisely, if we
put $B^\prime = (G \cdot B) \cup (G \cdot
B^{-1}),$ then it can be checked that the
following binary relation on $\bar{\mathbf{x}}$
is the smallest equivariant equivalence relation
containing $B$:
\begin{equation*} 
\Delta ~\cup~ B^\prime ~\cup~ (B^\prime \circ
B^\prime) ~\cup~ (B^\prime \circ B^\prime \circ
B^\prime) ~\cup~ \cdots.
\end{equation*}
For a binary relation $B$ on $\bar{\mathbf{x}},$
we define the following evaluation map:
\begin{equation*}
\ev_B : \Psi_F \longrightarrow \prod_{(\bar{x},
\bar{x}^\prime) \in B} ~ \iso \big( F_{\bar{x}},
~ F_{\bar{x}^\prime} \big), \quad \psi
\longmapsto \prod_{(\bar{x}, \bar{x}^\prime) \in
B} \psi(\bar{x}, \bar{x}^\prime).
\end{equation*}

\begin{lemma}
\label{lemma: evaluation of ptwise clutching} For
a binary relation $B$ on $\bar{\mathbf{x}},$ if
the smallest equivariant equivalence relation
containing $B$ is equal to $\bar{\mathbf{x}}
\times \bar{\mathbf{x}},$ then the map $\ev_B$ is
homeomorphic to its image.
\end{lemma}

\begin{proof}
First, we prove that $\ev_B$ is injective. The
map $\ev_B$ is injective if and only if any two
maps $\psi,$ $\psi^\prime \in \Psi_F$ identical
on $B$ are the same map. Hence to show that the
map $\ev_B$ is injective, we only have to show
that each $\psi$ is determined by $\psi(\bar{x},
\bar{x}^\prime)$'s for $(\bar{x}, \bar{x}^\prime
) \in B.$ However, this is basically the same
with the proof of Example \ref{example:
evaluation} because the smallest equivariant
equivalence relation containing $B$ is equal to
the whole set $\bar{\mathbf{x}} \times
\bar{\mathbf{x}}$ by the assumption. So, we
obtain a proof for injectivity.

Second, we show that $\ev_B$ is continuous. By
definition, the map $\ev_B$ is the composition of an
inclusion and an projection as follows:
\[
\Psi_F \hookrightarrow \prod_{(\bar{x},
\bar{x}^\prime) \in \bar{\mathbf{x}} \times
\bar{\mathbf{x}}} ~ \iso \big( F_{\bar{x}}, ~
F_{\bar{x}^\prime} \big) \longrightarrow
\prod_{(\bar{x}, \bar{x}^\prime) \in B} ~ \iso
\big( F_{\bar{x}}, ~ F_{\bar{x}^\prime} \big).
\]
Since the inclusion and the projection are continuous,
the map $\ev_B$ is continuous.

Last, we prove that $\ev_B^{-1}$ defined on the
image of $\ev_B$ is continuous. For simplicity,
we will give a proof only for the case in Example
\ref{example: evaluation}. In the case, the
binary relation $B = \{ (\bar{x}_0, \bar{x}_1)
\}$ satisfies the assumption of the lemma, and
the image of $\ev_B,$ denoted by $\im \ev_B,$ is
contained in $\iso ( F_{\bar{x}_0}, ~
F_{\bar{x}_1} ).$ We will express $\ev_B^{-1}$ as
the following composition of an inclusion and
some map $s$:
\[
\im \ev_B \hookrightarrow  \iso \big(
F_{\bar{x}_0}, ~ F_{\bar{x}_1} \big)
\overset{s}{\longrightarrow} \prod_{(\bar{x},
\bar{x}^\prime) \in \bar{\mathbf{x}} \times
\bar{\mathbf{x}}} ~ \iso \big( F_{\bar{x}}, ~
F_{\bar{x}^\prime} \big).
\]
Denoting $\iso ( F_{\bar{x}_i}, ~ F_{\bar{x}_j}
)$ by $\iso_{i, j},$ we define $s$ as follows:
\begin{align*}
\iso_{0, 1} ~ \rightarrow ~ & \iso_{0, 0} \times
\iso_{0, 1} \times \iso_{0, 2} \times \iso_{1, 0}
\times \iso_{1, 1} \times \iso_{1, 2} \times \iso_{2,
0} \times
\iso_{2, 1} \times \iso_{2, 2}, \\
A ~ \mapsto ~ & \Big( ~ \id ~ , ~ A ~ , ~ a^{-1}
A^{-1} a ~ , ~ A^{-1} ~ , ~ \id ~ , ~ a A a^{-1} ~ , ~
a^2 A a^{-2} ~ , ~ a A^{-1} a^{-1} ~ , ~ \id ~ \Big).
\end{align*}
This is nothing but piecing together formulas on $\psi
( \bar{x}_i, \bar{x}_i ),$ $\psi ( \bar{x}_i,
\bar{x}_{i+1} ),$ $\psi ( \bar{x}_{i+1}, \bar{x}_i )$
for $i \in \Z_3$ in Example \ref{example: evaluation},
and we can check that the above composition is equal
to the function $\ev_B^{-1}.$ So, $\ev_B^{-1}$ is
continuous because $s$ is continuous. Therefore, we
obtain a proof.
\end{proof}

We give some examples of this lemma.

\begin{example}
\begin{enumerate}
\item Assume that $\Z_m = \langle a \rangle$
acts on $\bar{\mathbf{x}} = \{ \bar{x}_i | i \in \Z_m
\}$ as follows:
\[
a \cdot \bar{x}_i = \bar{x}_{i+1} \quad \text{
for each }i \in \Z_m.
\]
Then, each $\psi \in \Psi_F$ is determined by
$\psi ( \bar{x}_0, \bar{x}_1 ).$
\item Assume that the order $2m$ dihedral group
\[
\D_m = \langle a, b ~ | ~ a^m = \id, ~ b^2 = \id,
~ b a b^{-1} = a^{-1} \rangle
\]
acts on $\bar{\mathbf{x}} = \{ \bar{x}_i | i \in
\Z_m \}$ as follows:
\[
a \cdot \bar{x}_i = \bar{x}_{i+1} \text{ and } b
\cdot \bar{x}_i = \bar{x}_{-i} \quad \text{ for
each }i \in \Z_m.
\]
Then, each $\psi \in \Psi_F$ is determined by
$\psi ( \bar{x}_0, \bar{x}_1 ).$
\item Assume that the order 4 dihedral group $\D_2$ acts on
$\bar{\mathbf{x}} = \{ \bar{x}_i | i \in \Z_4 \}$ as
follows:
\[
a \cdot \bar{x}_i = \bar{x}_{i+2} \text{ and } b
\cdot \bar{x}_i = \bar{x}_{-i+1} \quad \text{ for
each }i \in \Z_4.
\]
Then, each $\psi \in \Psi_F$ is determined by its
evaluation on
\[
B = \{ (\bar{x}_0, \bar{x}_1), ~ (\bar{x}_0,
\bar{x}_3) \}.
\]
\qed
\end{enumerate}
\end{example}

Next, we recall Example \ref{example: tetrahedron}
once again to explain motivation for restriction of an
equivariant pointwise clutching map by an example.

\begin{example}[Continued from Exercise
\ref{example: tetrahedron}] \label{example:
restriction} To glue the pullback vector bundle
$\pi^* E,$ we focused on how to glue $(\pi^*
E)_{\pi^{-1}(x)}$ for any point $x$ in an edge of
a regular tetrahedron $X.$ If we have an
understanding of all equivariant pointwise
clutching maps with respect to $(\pi^*
E)_{\pi^{-1}(x)}$ for each $x,$ then we could
pick one in each $\Psi_{(\pi^* E)_{\pi^{-1}(x)}}$
and collect them to form a continuous system of
equivariant pointwise clutching maps. But the
meaning of `continuous' is not clear at the
present stage because we can not compare two
equivariant pointwise clutching maps contained in
$\Psi_{(\pi^* E)_{\pi^{-1}(x)}}$'s for two
different $x$'s. Let us explain this.
\begin{figure}[ht]
\begin{center}
\begin{pspicture}(-4,0)(5,3.5) \footnotesize

\pspolygon[fillstyle=solid,fillcolor=lightgray](1.5,
1)(3.25, 0.75)(4,1.25)(2.75,2.5)(1.5, 1)
\psline[linestyle=dotted](1.5, 1)(4,1.25)
\psline(3.25, 0.75)(2.75,2.5)

\pspolygon[fillstyle=solid,fillcolor=lightgray](-1.5,0.75)(-0.75,1.25)(-3.25,1)(-1.5,0.75)
\pspolygon[fillstyle=solid,fillcolor=lightgray](-0.75,1.5)(-2,2.75)(-3.25,1.25)(-0.75,1.5)
\pspolygon[fillstyle=solid,fillcolor=lightgray](-1,
0.75)(-0.25,1.25)(-1.5,2.5)(-1,0.75)
\pspolygon[fillstyle=solid,fillcolor=lightgray](-2,0.75)(-2.5,2.5)(-3.75,1)(-2,0.75)

\psline[arrowsize=5pt]{->}(0.25,1.75)(1.25,1.75)

\psdots(-2.5,2.5)(-2,2.75)(-1.5,2.5)(2.75,2.5)

\psdots(2.85,2.15)(-2.4,2.15)(-1.4,2.15)

\uput[u](0.75,1.75){$\pi$} \uput[u](3,0){$X$}
\uput[u](-1.75,0){$\overline{X}$}

\uput[l](-2.5,2.5){$\bar{x}_0$}
\uput[u](-2,2.75){$\bar{x}_2$}
\uput[r](-1.5,2.5){$\bar{x}_1$}
\uput[u](2.75,2.5){$x$}

\uput[dl](2.95,2.15){$x^\prime$}
\uput[dl](-2.3,2.15){$\bar{x}_0^\prime$}
\uput[dl](-1.3,2.15){$\bar{x}_1^\prime$}
\end{pspicture}
\end{center}
\caption{\label{figure: restriction} Preimages of a
vertex and its nearby non-vertex point in an edge}
\end{figure}
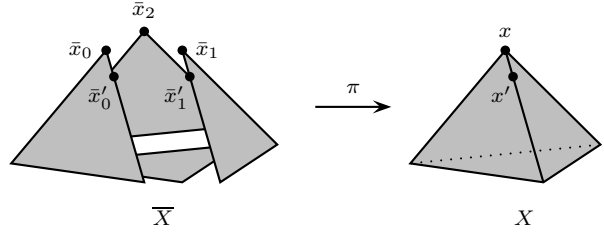
For two nearby points $x \ne x^\prime$ in an edge
of $X,$ let $\psi$ and $\psi^\prime$ be
equivariant pointwise clutching maps with respect
to $(\pi^* E)_{\pi^{-1}(x)}$ and $(\pi^*
E)_{\pi^{-1}(x^\prime)},$ respectively. If both
$x$ and $x^\prime$ are in the interior of the
edge, it would be conceivable to compare $\psi$
with $\psi^\prime$ because $(\pi^*
E)_{\pi^{-1}(x)}$ and $(\pi^*
E)_{\pi^{-1}(x^\prime)}$ are $G$-isomorphic and
we can identify them through a suitable
$G$-isomorphism. But if $x$ is a vertex and
$x^\prime$ is not, the situation is different
because cardinalities of their preimages under
$\pi$ are different, see Figure \ref{figure:
restriction}. To solve this problem, we restrict
the map $\psi$ to the smaller subset $\{
\bar{x}_0, \bar{x}_1 \} \times \{ \bar{x}_0,
\bar{x}_1 \}.$ Then, the restricted map will turn
out to be an equivariant pointwise clutching map
with respect to $(\pi^* E)_{\{ \bar{x}_0,
\bar{x}_1 \}}.$ And then we could compare the map
so obtained with $\psi^\prime$ because $(\pi^*
E)_{\{ \bar{x}_0, \bar{x}_1 \}}$ and $(\pi^*
E)_{\pi^{-1}(x^\prime)}$ are $G$-isomorphic and
we can identify them through a suitable
$G$-isomorphism. \qed
\end{example}

We prove a lemma on restriction of an equivariant
pointwise clutching map.

\begin{lemma}  \label{lemma: restricted pointwise clutching}
For a subset $\bar{\mathbf{x}}^\prime \subset
\bar{\mathbf{x}}$ and the maximal subgroup $K \subset
G$ preserving $\bar{\mathbf{x}}^\prime,$ the map
\begin{equation*}
\ev_{\bar{\mathbf{x}}^\prime \times
\bar{\mathbf{x}}^\prime} : \Psi_F \rightarrow
\Psi_{F_{\bar{\mathbf{x}}^\prime}}, \quad \psi \mapsto
\ev_{\bar{\mathbf{x}}^\prime \times
\bar{\mathbf{x}}^\prime} ( \psi )
\end{equation*}
is well-defined, and we have a $K$-isomorphism
\begin{equation*}
\res_K^G ~ F / \psi \quad \cong \quad
F_{\bar{\mathbf{x}}^\prime} \Big/
\ev_{\bar{\mathbf{x}}^\prime \times
\bar{\mathbf{x}}^\prime} ( \psi )
\end{equation*}
for any $\psi \in \Psi_F.$
\end{lemma}

\begin{proof}
Note that $F_{\bar{\mathbf{x}}^\prime}$ is a
$K$-vector bundle. For any $\psi \in \Psi_F,$ we
consider the map $p_\psi : F \rightarrow F/\psi.$
Restricting $p_\psi$ to
$F_{\bar{\mathbf{x}}^\prime},$ we obtain an
equivariant fiberwise isomorphism
\[
p_\psi|_{F_{\bar{\mathbf{x}}^\prime}} :
F_{\bar{\mathbf{x}}^\prime} \rightarrow \res_K^G
F/\psi.
\]
By Lemma \ref{lemma: define clutching}, we can
construct an equivariant pointwise clutching map
$\psi^\prime$ in
$\Psi_{F_{\bar{\mathbf{x}}^\prime}}$ such that
the following diagram commutes:
\begin{equation*}
\SelectTips{cm}{} \xymatrix{
F_{\bar{\mathbf{x}}^\prime}
\ar[r]^-{p_\psi|_{F_{\bar{\mathbf{x}}^\prime}}}
\ar[d]_-{p_{\psi^\prime}} & \qquad \res_K^G
F/\psi \\
F_{\bar{\mathbf{x}}^\prime} \big/ \psi^\prime \quad
\ar@{->}[ur]_-{\cong} }.
\end{equation*}
By definition of $\psi^\prime,$ the map
$\psi^\prime$ is equal to
$\ev_{\bar{\mathbf{x}}^\prime \times
\bar{\mathbf{x}}^\prime} ( \psi ),$ so the map
$\ev_{\bar{\mathbf{x}}^\prime \times
\bar{\mathbf{x}}^\prime} ( \psi )$ is contained
in $\Psi_{F_{\bar{\mathbf{x}}^\prime}}.$ And, the
map $\ev_{\bar{\mathbf{x}}^\prime \times
\bar{\mathbf{x}}^\prime}$ is well-defined. Also,
we obtain the $K$-isomorphism by the diagram.
Therefore, we obtain a proof.
\end{proof}

\end{document}